\documentclass[11pt,a4paper]{article}
\usepackage{ae,aecompl}
\usepackage{amssymb,amsmath,amsthm}
\usepackage[bookmarksnumbered=true,pdfstartview=FitH]{hyperref}
\usepackage[english]{babel}
\usepackage{fancyhdr}
\numberwithin{equation}{section}
\allowdisplaybreaks[2]

\newtheorem{thm}{Theorem}[section]
\newtheorem{lemma}[thm]{Lemma}
\newtheorem{prop}[thm]{Proposition}
\newtheorem{cor}[thm]{Corollary}
\newtheorem{rem}[thm]{Remark}
\newtheorem{df}[thm]{Definition}

\newcommand{\A}{\mathcal{A}}
\newcommand{\B}{\mathcal{B}}
\newcommand{\K}{\mathcal{K}}
\newcommand{\HH}{\mathcal{H}}
\newcommand{\M}{\mathcal{M}}
\newcommand{\R}{\mathbb{R}}
\newcommand{\C}{\mathbb{C}}
\newcommand{\N}{\mathbb{N}}
\newcommand{\Z}{\mathbb{Z}}
\newcommand{\Q}{\mathbb{H}}
\newcommand{\D}{D\mkern-11.5mu/\,}
\newcommand{\inner}[1]{\left<#1\right>}
\newcommand{\qiso}{\smash[t]{\mathrm{QISO}}\rule{0pt}{8pt}^+_J}
\newcommand{\qisot}{\widetilde{\rule{0pt}{7pt}\smash[t]{\mathrm{QISO}}}\rule{0pt}{8pt}^+_J}
\newcommand{\qisor}{\smash[t]{\mathrm{QISO}}\rule{0pt}{8pt}^+_{\R}}
\newcommand{\qisotr}{\widetilde{\rule{0pt}{7pt}\smash[t]{\mathrm{QISO}}}\rule{0pt}{8pt}^+_{\R}}
\newenvironment{mat}{\bigg(\!\!\begin{array}{cc}}{\end{array}\!\!\bigg)}
\newenvironment{mfour}{\begin{small}\begin{pmatrix}}{\end{pmatrix}\end{small}}
\newcommand{\mc}[1]{\mathcal{#1}}
\renewcommand{\bar}{\overline}
\newcommand{\tr}{\mathrm{Tr}}
\newcommand{\punto}{\item[\raisebox{1.5pt}{\small$\blacktriangleright$}]}
\newcommand{\catA}{\mathfrak{C}_J}
\newcommand{\catB}{\mathfrak{C}_{J,\R}}

\title{Quantum Isometries of the finite noncommutative\\ geometry of the Standard Model}

\date{\empty}

\author{~\\
\hspace*{-.4cm}Jyotishman Bhowmick$^{\,*}$,
Francesco D'Andrea$^{\,\ddag}$,
Ludwik D{\k a}browski$^{\,\ddag}$\\[20pt]
\hspace*{-.9cm}\textit{\small $^*$ Abdus Salam International Center for Theoretical Physics (ICTP), Strada Costiera 11, I-34151 Trieste, Italy}\\
\hspace*{-.9cm}\textit{\small $^\ddag$ Scuola Internazionale Superiore di Studi Avanzati (SISSA), via Bonomea 265, I-34136 Trieste, Italy}}

\begin{document}

\maketitle

\begin{abstract}\noindent%
We compute the quantum isometry group of the finite noncommutative geometry $F$ describing the internal degrees of freedom in the Standard Model of particle physics.
We show that this provides genuine quantum symmetries of the spectral triple corresponding to $M\times F$, where $ M $ is a compact spin manifold.
We also prove that the bosonic and fermionic part of the spectral action are preserved by these symmetries.
\end{abstract}

\bigskip


\section{Introduction}
In modern  theoretical physics, symmetries play a fundamental role in determining the dynamics of a theory. In the two foremost examples,
namely General Relativity and the Standard Model of elementary particles, the dynamics is dictated by invariance under diffeomorphisms and
under local gauge transformations respectively.
As a way to unify external (i.e.~diffeomorphisms) and internal (i.e.~local gauge) symmetries, Connes and Chamseddine proposed a model from Noncommutative Geometry \cite{Con94} based on the product of the canonical commutative spectral triple of a 
compact Riemannian spin manifold $M$ and a finite dimensional noncommutative one, describing an ``internal'' finite noncommutative space $F$ \cite{CC97,CC08,Con06,CM08}. In this picture, diffeomorphisms are realized as outer automorphisms of the algebra, while inner automorphisms correspond to the gauge transformations. Inner fluctuations of the Dirac operator are divided in two classes: the $1$-forms coming from commutators with the Dirac operator of $M$ give the gauge bosons, while the $1$-forms coming from the Dirac operator of $F$ give the Higgs field. The gravitational and bosonic part $S_b$ of the action is encoded in the spectrum of the gauged Dirac operator, which is invariant under isometries of the Hilbert space. The fermionic part $S_f$ is also defined in terms of the spectral data.
The result is an Euclidean version of the Standard Model minimally coupled to gravity (cf.~\cite{CM08} and references therein).

In his ``Erlangen program'', Klein linked the study of geometry with the analysis of its group of symmetries. Dealing with quantum geometries, it is natural to study quantum symmetries.
The idea of using quantum group symmetries to understand the conceptual significance of the finite geometry $F$ is mentioned in a final remark by Connes in \cite{Con98}.
Preliminary studies on the Hopf-algebra level appeared in \cite{Kas95,Coq97,DNS98}.
Following Connes' suggestion, quantum automorphisms of finite-dimensional complex \mbox{$C^*$-algebras} were introduced by Wang
in \cite{Wan98a,Wan98b} and later the quantum permutation groups of finite sets and graphs have been studied by a number of
mathematicians, see e.g.~\cite{Ban05a,Ban05b,Bic03,Sol10}. These are compact quantum groups in the sense of Woronowicz \cite{Wor95}.
The notion of compact quantum symmetries for ``continuous'' mathematical structures, like commutative and noncommutative manifolds (spectral triples),
first appeared in \cite{Gos07}, where quantum isometry groups were defined in terms of a Laplacian, followed by the definition of ``quantum groups of orientation preserving isometries'' based on the theory of spectral triples
in \cite{BG09}, and on spectral triples with a real structure in \cite{Gos10}. Computations of these compact quantum groups  were done for several examples, including the tori, spheres, Podle\'s quantum spheres, and Rieffel deformations of compact Riemannian spin manifolds.
For these studies we refer to \cite{BGS08,BG09,BG10,BG09b,BS10} and references therein.

The finite noncommutative geometry $F=(A_F,H_F,D_F,\gamma_F,J_F)$ describing the internal
space of the Standard Model is given by a unital real spectral triple over the finite-dimensional
real $C^*$-algebra $A_F=\C\oplus\Q\oplus M_3(\C)$, with $\Q$ the field of quaternions. Let $B_F\subset\B(\HH)$
be the smallest complex $C^*$-algebra containing $A_F$ as a real $C^*$-subalgebra.
In this article we first compute the quantum group of orientation and real structure preserving isometries of the spectral triple $(B_F,H_F,D_F,\gamma_F,J_F)$;
next we show that this quantum symmetry can be extended to get quantum isometries of the product of this
spectral triple with the canonical spectral triple of $M$. Thus, we have genuine quantum symmetries of the
full spectral triple of the Standard Model. Moreover these quantum symmetries preserves the spectral action
in a suitable sense. Finally we compute the maximal quantum subgroup of the quantum isometry group whose
coaction is a quantum automorphism of the real $C^*$-algebra $A_F$.

The plan of this article is as follows.
We start by recalling in Sec.~\ref{sec:cqg} some basic definitions and facts about compact quantum groups and quantum isometries.
In Sec.~\ref{sec:tre} we introduce the spectral triple $F$ and state the main result. Since quantum groups, coactions, etc.~are defined
in the framework of complex ($C^*$-)algebras, we replace $A_F$ by $ B_F $ and compute the quantum isometry group
of the latter in the sense of \cite{Gos10}. As shown in Sec.~\ref{sec:iso}, this is given by the free product $C(U(1))*A_{\mathrm{aut}}(M_3(\C))$,
where $A_{\mathrm{aut}}(M_n(\C))$ is Wang's quantum automorphism group of $M_n(\C)$
\cite{Wan98a}. In Sec.~\ref{sec:5}, we discuss the invariance of the spectral action under quantum isometries.
In Sec.~\ref{sec:6} we explain how the result changes if we work with real instead of complex algebras. The final section deals with the proof of the main result, that is, Proposition \ref{prop:main}.

Throughout the paper, by the symbol $\otimes_{\mathrm{alg}}$ we always mean the algebraic tensor product over $\C$,
by $\otimes$ the minimal tensor product of complex $C^*$-algebras or the completed tensor product of
Hilbert modules over complex $C^*$-algebras. The symbol $\otimes_{\R}$ denotes the tensor product
over the real numbers.
Unless otherwise stated, all algebras are assumed to be unital complex
associative involutive algebras.
We denote by $\mathcal{N}^*$ the set of all bounded linear functionals $\mathcal{N}\to\C$
on the normed linear space $\mathcal{N}$,
by $\mc{M}(\A)$ the  multiplier algebra of the complex $C^*$-algebra $\A$,
by $\mc{L}(\HH)$ the adjointable operators on the Hilbert module $\HH$
and by $\K(\HH)$ the compact operators on the Hilbert space $\HH$.
For a unital complex $C^*$-algebra $\A$, we implicitly use the identification of $\mc{M}(\K(\HH)\otimes\A)$ with
the set of all adjointable operators on the Hilbert $\A$-module $\HH\otimes\A$.
By abelianization of $\A$ we mean the quotient of $\A$ by its commutator $C^*$-ideal.
Given a matrix $u$ with entries $u_{ij}$ in a $C^*$-algebra $\A$, we denote by
$u_{ij}^*=(u_{ij})^*$ the conjugate of the element $u_{ij}$, and by $(u^*)_{ij}=u_{ji}^*$ the
entry $(i,j)$ of the adjoint matrix $u^*$.
Lastly, we want to attract the reader's attention to a choice of notation.  The notation $ \qisot $ used in this article is the same as $ \widetilde{\rule{0pt}{7pt}\smash[t]{\mathrm{QISO}}}\rule{0pt}{8pt}^+_{\mathrm{real}} $ of \cite{Gos10}. We do this to avoid confusion with the newly defined object $ \qisotr $ of Section \ref{sec:6} in the context of quantum isometries of real $C^*$-algebras.

\section{Compact quantum groups and quantum isometries}\label{sec:cqg}

\subsection{Some generalities on Compact Quantum Groups}

We begin by recalling the definition of compact quantum groups and their coactions from  \cite{Wor87,Wor95}. We shall use most of the terminology of
\cite{Wan95}, for example Woronowicz $C^*$-subalgebra, Woronowicz
$C^*$-ideal, etc., however with the exception that
Woronowicz \mbox{$C^*$-algebras} will be called
compact quantum groups, and we will not use
the term compact quantum groups for the dual objects as done in
\cite{Wan95}.

\begin{df}
A \emph{compact quantum group} (to be denoted by CQG from now on) is a pair $(Q,\Delta)$ given by a complex unital \mbox{$C^*$-algebra}
$Q$ and a unital $C^*$-algebra morphism $\Delta:Q\to Q\otimes Q$ such that
\begin{itemize}
\item[i)]
$\Delta$ is coassociative, i.e.
$$
(\Delta\otimes id)\circ\Delta=(id\otimes\Delta)\circ\Delta
$$
as equality of maps $Q\to Q\otimes Q\otimes Q$;
\item[ii)]
$\mathrm{Span}\bigl\{(a\otimes 1_Q)\Delta(b)\,\big|\,a,b\!\in\! Q\bigr\}$ and
$\,\mathrm{Span}\bigl\{(1_Q\otimes a)\Delta(b)\,\big|\,a,b\!\in\! Q\bigr\}$
are norm-dense in $Q\otimes Q$.
\end{itemize}
\end{df}

\noindent
For $Q=C(G)$, where $G$ is a compact topological group,
conditions i) and ii) correspond to the associativity and the cancellation
property of the product in $G$, respectively.

\begin{df}
A \emph{unitary corepresentation} of a compact quantum group $(Q,\Delta) $ on a Hilbert space $\HH$ is a unitary element $U\in\M(\K(\HH)\otimes Q)$
satisfying
$$
(id\otimes \Delta ) U = U_{(12)} U_{(13)} \;,
$$
where we use the standard leg numbering notation (see e.g.~\cite{MD98}).
\end{df}

\noindent
If $Q=C(G)$, $U$ corresponds to a strongly continuous unitary representation of $G$.

For any compact quantum group $Q$ (see \cite{Wor87,Wor95}), there always exists a canonical dense $*$-subalgebra $Q_0\subset Q$ which is spanned by  the matrix coefficients of the finite dimensional unitary corepresentations of $Q$ and two maps
$\epsilon : Q_0 \to \C$ (counit) and  $\kappa : Q_0 \to Q_0$ (antipode) which make $Q_0$ a Hopf $*$-algebra.

\begin{df}
A \emph{Woronowicz $C^*$-ideal} of a CQG $(Q,\Delta)$ is a $C^*$-ideal $I$ of $Q$ such that $\Delta(I)\subset\ker(\pi_I\otimes\pi_I)$, where $\pi_I:Q\to Q/I$ is the projection map.
The quotient $Q/I$ is a CQG with the induced coproduct.
\end{df}

If $Q=C(G)$ are continuous functions on a compact topological group $G$, closed subgroups of $G$ correspond to the quotients of $Q$ by its Woronowicz $C^*$-ideals. While quotients $Q/I$ give ``compact quantum subgroups'', $C^*$-subalgebras $Q'\subset Q$ such that $\Delta(Q')\subset Q'\otimes Q'$ describe ``quotient quantum groups''.

\begin{df}\label{def:alpha}
We say that a CQG $(Q,\Delta)$ coacts on a unital $C^*$-algebra $\A$
if there is a unital $C^*$-homomorphism (called a \emph{coaction})
$\alpha:\A\to\A\otimes Q$ such that:
\begin{itemize}
\item[i)] $(\alpha \otimes id) \alpha=(id\otimes \Delta) \alpha$,
\item[ii)]
$\mathrm{Span}\bigl\{\alpha(a)(1_\A \otimes b)
\,\big|\,a\in\A,\,b\in Q\bigr\}$ is norm-dense in $\A\otimes Q$.
\end{itemize}
The coaction is \emph{faithful} if any compact quantum group $Q'\subset Q$ coacting on $\A$
coincides with $Q$.
\end{df}

\noindent
It is well known (cf.~\cite{Pod95,Wan98a}) that condition (ii) in Def.~\ref{def:alpha} is equivalent to the existence of a norm-dense unital $*$-subalgebra $\A_0$ of $\A$ such that $\alpha(\A_0) \subset \A_0 \otimes_{\mathrm{alg}}Q_0$ and $(id\otimes \epsilon) \alpha=id$
on $\A_0$.
For later use, let us now recall the concept of universal CQGs $A_u(R)$ as defined in
\cite{DW96,Wan98b} and references therein.

\begin{df}\label{def:WangA}
For a fixed $n \times n$ positive invertible matrix $R$, $A_u(R)$ is the universal $C^*$-algebra generated
by $\{u_{ij},\,i,j=1,\ldots,n\}$ such that
$$
uu^*=u^*u=\mathbb{I}_n\;,\qquad
u^t(R\bar uR^{-1})=(R\bar uR^{-1})u^t=\mathbb{I}_n \;,
$$
where $u:=((u_{ij}))$, $u^*: = (( u^*_{ji} )) $ and
$\,\bar u:=(u^*)^t=((u_{ij}^*))$.
The coproduct $\Delta$ is given by
$$
\Delta(u_{ij})=\sum\nolimits_k u_{ik} \otimes u_{kj} \;.
$$
\end{df}

\noindent
Note that $u$ is a unitary corepresentation of $A_u(R)$ on $\C^n$.

The $A_u(R)$'s are universal in the sense that every compact \emph{matrix}
quantum group (i.e.~every CQG generated by the matrix entries of a finite-dimensional
unitary corepresentation) is a quantum subgroup of $A_u(R)$ for some $R>0$ \cite{Wan98b}.
It may also be noted that $A_u(R)$ is the universal object in the category of CQGs which admit a unitary corepresentation on $\C^n$ such that the adjoint coaction on the finite-dimensional $C^*$-algebra $M_n(\C)$ preserves  the functional $M_n(\C)\ni m \mapsto \mathrm{Tr}(R^t m)$ (see \cite{Wan99}).

We observe the following elementary fact which is going to be used in the sequel.

\begin{lemma}\label{lemma:trace}
Let $\HH=\C^n$, $n\in\N$ and $B\in M_n(\B)$ be
a matrix with entries in a unital $*$-algebra $\B$. Then
$$
( \tr_{\HH} \otimes {\rm id} ) \,B(L\otimes 1)B^*=\tr_{\HH}(L)\cdot 1_{\B}
$$
for any linear operator $L$ on $\HH$ if and only if $B^t\!$ is unitary.
\end{lemma}

A matrix $B$ (with entries in a unital $*$-algebra $\B$) such that
both $B$ and $B^t$ are unitary is called a \emph{biunitary}
\cite{BV09}. We remark that the CQG $A_u(n):=A_u(\mathbb{I}_n)$, called the
\emph{free quantum
unitary group}, is generated by the biunitary matrix $u$ given in Def.~\ref{def:WangA}.
We refer to \cite{Wan98b} for a detailed discussion on the structure and classification of such quantum groups.

The analogue of projective unitary groups was introduced in \cite{Ban97}
(see also Sec.~3 of \cite{BV09}). Let us recall the definition.

\begin{df}\label{def:Ban}
We denote by $PA_u(n)$ the $C^*$-subalgebra of $A_u(n)$ generated
by $\{(u_{ij})^*u_{kl}:\,i,j,$ $k,l=1,\ldots,n\}$.
This is a CQG with the coproduct induced from $A_u(n)$.
\end{df}

\begin{rem}\label{remarkprojectiveversion}
The projective version of any quantum subgroup of  $ A_u ( n ) $ can be defined similarly.
\end{rem}

In \cite{Wan98a}, Wang defines the quantum automorphism group of $M_n(\C)$,
denoted by $A_{\mathrm{aut}}(M_n(\C))$ to be the universal object in the
category of CQGs with a coaction on $M_n(\C)$ preserving the trace
(and with morphisms given by CQGs homomorphisms intertwining the
coactions). The explicit definition is in Theorem 4.1 of \cite{Wan98a}.
In the following proposition we recall
 Th{\'e}or{\`e}me 1(iv) of \cite{Ban97}
(cf.~also Prop.~3.1(3) of \cite{BV09}).

\begin{prop}[\cite{Ban97,BV09}]
\label{thm:Ban}
We have $PA_u(n)\simeq A_{\mathrm{aut}}(M_n(\C))$.
\end{prop}

\begin{df}\label{def:amalg}
We denote by $Q_n(n')$ the amalgamated free product of $n$ copies of $A_u (n')$ over the common
Woronowicz $ C^* $-subalgebra $ PA_u(n')$. This is the CQG generated by the matrix entries of
$n$ biunitary matrices $u_m$ ($m=1,\ldots,n$) of size $n'$, with relations
$$
(u_m^*)_{i,j}(u_m)_{k,l}=(u_{m'}^*)_{i,j}(u_{m'})_{k,l} \qquad\forall\; i,j,k,l=1,\ldots,n',\; m,m'=1,\ldots,n\,,
$$
and with standard matrix coproduct: $\Delta((u_m)_{ij})=\sum_{k=1}^{n'}(u_m)_{ik}\otimes (u_m)_{kj}$
for all $m=1,\ldots,n$.
\end{df}

The next lemma will be needed later on.

\begin{lemma}\label{lemma:6.10}
Let $Q$ be a CQG and $X,Y\in M_N(Q)$, $N\in\N$, be matrices with entries in $Q$
satisfying
$\Delta(X_{ik})=\sum_{j=1}^NX_{ij}\otimes X_{jk}$ and 
$\Delta(Y_{ik})=\sum_{j=1}^NY_{ij}\otimes Y_{jk}$.
Let $A\in M_N(\C)$.
Then the ideal $I\subset Q$ generated by the matrix entries
of the matrix $XA-AY$ is a Woronowicz $C^*$-ideal.
\end{lemma}
\begin{proof}
We now prove that $\Delta(I)\subset Q\otimes I+I\otimes Q\subset
\ker(\pi_I\otimes\pi_I)$, where $\pi_I:Q\to Q/I$ is the quotient map,
and hence $I$ is a Woronowicz $C^*$-ideal. Since $I$ is a (two-sided) ideal
and $\Delta$ a $C^*$-algebra homomorphism, it is enough to give the proof
for the generators $Z_{ij}:=\sum_{k=1}^N(X_{ik}A_{kj}-A_{ik}Y_{kj})$ of $I$.
The following algebraic identity holds
\begin{align*}
\Delta(Z_{il})&=\sum\nolimits_{j=1}^N\Delta(X_{ij}A_{jl}-A_{ij}Y_{jl})
\\
&=\sum\nolimits_{j,k=1}^N  X_{ij}\otimes X_{jk}A_{kl}
-A_{ij}Y_{jk}\otimes Y_{kl} \\
&=
\sum\nolimits_{j,k=1}^N  X_{ij}\otimes (X_{jk}A_{kl}-A_{jk}Y_{kl})
+(X_{ij}A_{jk}-A_{ij}Y_{jk})\otimes Y_{kl} \\
&=\sum\nolimits_{j=1}^N  (X_{ij}\otimes Z_{jl}+Z_{ij}\otimes Y_{jl}) \;.
\end{align*}
This concludes the proof.
\end{proof}

\subsection{Noncommutative Geometry and quantum isometries}
In noncommutative geometry, compact Riemannian spin manifolds are replaced by real spectral triples.
Recall that a \emph{unital spectral triple} $(\A,\HH,D)$ is the datum of:
a complex Hilbert space $\HH$, a  complex unital associative involutive
algebra $\A$ with a faithful unital $*$-representation $\pi:\A\to\B(\HH)$ (the representation
symbol is usually omitted), a (possibly unbounded) self-adjoint operator $D$ on $\HH$ with
compact resolvent and having bounded commutators with all $a\in\A$.
The canonical commutative example is given by $(C^\infty(M),L^2(M,S),\D)$,
where $C^\infty(M)$ is the algebra of complex-valued smooth functions on a compact Riemannian
spin manifold with no boundary, $L^2(M,S)$ is the Hilbert space of square integrable
spinors and $\D$ is the Dirac operator.

A spectral triple is \emph{even} if there is a $\Z_2$-grading $\gamma$ on $\HH$ commuting
with $\A$ and anticommuting with $D$. We will set $\gamma=1$ when the spectral triple is odd.

A spectral triple is \emph{real} if there is an antilinear isometry $J:\HH\to\HH$, called
the \emph{real structure}, such that
\begin{equation}\label{eq:yyy}
J^2=\epsilon 1\;,\qquad JD=\epsilon' DJ\;,\qquad
J\gamma=\epsilon''\gamma J\;,
\end{equation}
and
\begin{equation}\label{eq:real}
[a,JbJ^{-1}]=0\;,\qquad
[[D,a],JbJ^{-1}]=0\;,
\end{equation}
for all $a,b\in\A$~\footnote{Notice that in some examples, although not in the present case, the condition
\eqref{eq:real} has to be slightly relaxed, cf.~\cite{DDLW07,DDL06,DLPS05,DLS05}.}.
$\epsilon$, $\epsilon'$ and $\epsilon''$ are signs and determine the KO-dimension of the space \cite{Con95}.

For the finite part of the Standard Model $\epsilon = +1$, $\epsilon' = +1$, $\epsilon'' = - 1$
and the KO-dimension is $6$ \cite{CCM06}.
Imposing a few additional conditions, it is possible to reconstruct a compact
Riemannian spin manifold from any commutative real spectral triple \cite{Con08}.

In the example $(C^\infty(M),L^2(M,S),\D,J,\gamma)$ of the spectral triple associated to a compact Riemannian spin manifold $M$ with no boundary, there exists a covering group $\widetilde{G}$ of the group of orientation preserving isometries $ G $ of $M$ having a unitary representation $U$ on the Hilbert space of spinors $L^2(M,S)$ commuting with $ \D, J, \gamma $ whose adjoint action $\mathrm{Ad_U}$ on $\B(L^2(M,S))$ preserves the subalgebra $C^\infty(M)$. This picture is used to generalize the notion of isometries as follows (cf.~Def.~3 and 4 of \cite{Gos10}).

\begin{df}\label{def:qisot}
A compact quantum group $Q$ coacts by ``orientation and real structure preserving isometries'' on the spectral triple
$(\A,\HH,D,\gamma,J)$ if there is a unitary corepresentation
$U\in\M(\K(\HH)\otimes Q)$ such that
\begin{subequations}
\begin{align}
& \hspace{-3.5cm}
\text{$U$ commutes with $D\otimes 1$ and $\gamma\otimes 1$;}\label{eq:cond1} \\[4pt]
& \hspace{-3.5cm}
\text{$(J\otimes *) U(\xi\otimes 1_Q) = U(J\xi\otimes 1_Q)$ for all $\xi\in\HH$;} \label{eq:cond2} \\[4pt]
& \hspace{-3.5cm}
\text{$(\mathrm{id} \otimes\varphi)\mathrm{Ad_U} (a)\in\A''$ for all $a\in \A$ and every state $\varphi$ on $Q$,} \label{eq:cond3}
\end{align}
\end{subequations}
where $\mathrm{Ad_U} =U(\,.\,\otimes 1_Q)U^*$ is the adjoint coaction and $\A''$ is the double commutant of $\A$.
\end{df}

Note that in Definition 4 of \cite{Gos10} two antilinear operators $J$ and $\tilde J$ appear. $\tilde J$ is a generalized real structure (it is not assumed to be
an isometry) and $J$ is its antiunitary part. As in the case of this article the real structure is an antilinear isometry $J$ and $\tilde J$ coincide and hence our definition is a particular instance of Definition 4 of \cite{Gos10}.

We end this section by recalling Theorem 1 of \cite{Gos10}. Let $(\A,\HH,D,\gamma,J)$ be a real spectral triple with $ \epsilon^{\prime} = 1 $
and $\catA$ be the category with objects $ ( Q, U ) $ as in Definition \ref{def:qisot} and morphisms given by CQG morphisms intertwining the corresponding corepresentations. 
We recall that an object $(Q,U)$ in the category $\catA$ is said to be a sub-object of $(Q_0,U_0)$ in the same category if there exists a
CQG morphism $\varphi:Q_0\to Q$ such that $(id\otimes\varphi)(U_0)=U$. An object $(Q_0,U_0)$ is \emph{universal} if for any other object $(Q,U)$ in $\catA$
there exists unique such $\varphi$.

\begin{thm}[\cite{Gos10}]
The category $\catA$ has a universal object
denoted by $\widetilde{\rule{0pt}{7pt}\smash[t]{\mathrm{QISO}}}\rule{0pt}{8pt}^+(\A,\HH,D,\gamma,J)$ (or simply $\qisot(D)$)
whose unitary corepresentation, say $ U_0, $ is faithful.  The \emph{quantum isometry group}, denoted
by $\mathrm{QISO}^+(\A,\HH,D,\gamma,J)$ (or simply $\qiso(D)$),
is given by the Woronowicz $C^*$-subalgebra of $\qisot(D)$
generated by the elements 
$\inner{\xi\otimes 1, \mathrm{Ad_{U_0}} ( a )(\eta\otimes 1)}$, where $a\in\A$,
$\xi,\eta\in\HH$ and $\inner{\,,\,}$ is the $\qisot(D)$-valued inner
product on the Hilbert
module $\HH\otimes \qisot(D)$ (cf.~Def.~5 in \cite{Gos10}).
\end{thm}

$\qisot(D)$ is the quantum analogue of
the covering $\widetilde{G}$ of the classical group $G$ of orientation preserving isometries of a spin manifold $M$. It's projective version (in the sense of Sec.~3 of \cite{BV09}) is the quantum group $\qiso ( D )$, which is the quantum analogue of $G$.

\section{Quantum isometries of the internal non-commutative space of the Standard Model}\label{sec:tre}
\subsection{The finite non-commutative space $F$}\label{sec:SM}

The spectral triple $(A_F,H_F,D_F,\gamma_F,J_F)$
describing the internal space $F$ of the Standard Model is defined as follows
(cf.~\cite{CM08} and references therein).
The algebra $A_F$ is
\begin{equation}\label{eq:AF}
A_F:=\C\oplus\Q\oplus M_3(\C) \;,
\end{equation}
where we identify $\Q$ with the real subalgebra of
$M_2(\C)$ with elements
\begin{equation}\label{eq:XXX}
q=\begin{mat}
\alpha & \beta \\
-\bar\beta & \bar\alpha
\end{mat}
\end{equation}
for $\alpha,\beta\in\C$ (cf.~Cayley-Dickson construction).

Let us denote by $\C[v_1,\ldots,v_k]\simeq\C^k$ the vector space with basis
$v_1,\ldots,v_k$. For our convenience, we adopt the following notation for the
Hilbert space $H_F$. It can be written as a tensor product
$$
H_F:=\C^2\otimes\C^4\otimes\C^4\otimes\C^n \;,
$$
where, in the notations of~\cite{CM08}, we have
\begin{itemize}\itemsep=0pt
\item[i)] the first two factors $\C^2\otimes\C^4$ with
$$
\C^2=\C[\uparrow,\downarrow] \;,\qquad
\C^4=\C[\ell,\{q_c\}_{c=1,2,3}]  \;,
$$
where $\uparrow$ and $\downarrow$ stand for weak isospin up and down,
$\ell$ and $q_c$ stand for lepton and quark of color $c$ respectively.
These may be combined into
$$
\C^8=\C[\nu,e,\{u_c,d_c\}_{c=1,2,3}] \;,
$$
where
$\nu$ stands for ``neutrino'', $e$ for ``electron'', $u_c$ and $d_c$ for
quarks with weak isospin $+1/2$ and $-1/2$ respectively and of color $c$.
Explicitly, the isomorphism $\C^2\otimes\C^4\to\C^8$
is the map
$$
{\uparrow}\otimes\ell\mapsto \nu\;,\quad
{\downarrow}\otimes\ell\mapsto e\;,\quad
{\uparrow}\otimes q_c\mapsto u_c\;,\quad
{\downarrow}\otimes q_c\mapsto d_c\;.
$$

\item[ii)] a factor
$$
\C^4=\C[p_L,\bar p_R,\bar p_L,p_R] \;,
$$
where $L,R$ stand for the two chiralities, $p$ for ``particle'' and
$\bar p$ for ``antiparticle'';

\item[iii)] a factor $\C^n$ since each particle
comes in $n$ generations. Presently only $3$ generations have been observed,
but for the sake of generality we will work with an arbitrary
$n\geq 3$.
\end{itemize}
From a physical point of view, rays (lines through the origin) of $H_F$ are states describing the internal degrees of freedom of the elementary fermions.
The \emph{charge conjugation} $J_F$ changes a particle into its antiparticle,
and is the composition of the componentwise
complex conjugation on $H_F$ with the linear operator
\begin{equation}\label{eq:J0}
J_0:=
1\otimes 1\otimes {\footnotesize
\begin{pmatrix}
0 & 0 & 1 & 0 \\
0 & 0 & 0 & 1 \\
1 & 0 & 0 & 0 \\
0 & 1 & 0 & 0
\end{pmatrix}}
\otimes 1 \;.
\end{equation}
The grading is
$$
\gamma_F:=
1\otimes 1\otimes
\mathrm{diag}(1,1,-1,-1)
\otimes 1 \;.
$$
The element $a=(\lambda,q,m)\in A_F$ (with $\lambda\in\C$, $q\in\Q$ and $m\in M_3(\C)$) is represented by
\begin{align}
\pi(a) &=
q\otimes 1\otimes e_{11}\otimes 1 +
{\small\begin{mat}
\lambda & 0 \\ 0 & \bar\lambda
\end{mat}}\otimes 1
\otimes e_{44}\otimes 1 \notag \\[2pt]  & \qquad +
1\otimes
{\small
\begin{pmatrix}
\lambda & 0 & 0 & 0 \\
0 \\
0 && {\normalsize m} \\
0
\end{pmatrix}}
\otimes (e_{22}+e_{33})\otimes 1 \;,\label{eq:repPi}
\end{align}
where $m$ is a $3\times 3$ block and
$\{e_{ij}\}_{i,j=1,\ldots,k}$ is the canonical basis
of $M_k(\C)$ ($e_{ij}$ is the matrix with $1$ in the $(i,j)$-th
position and $0$ everywhere else). In particular, in \eqref{eq:repPi}
$e_{11}$ projects on the space $\C[p_L]$ of particles with left chirality,
$e_{22}$ on $\C[\bar p_R]$, $e_{33}$ on $\C[\bar p_L]$ and
$e_{44}$ on $\C[p_R]$.

The Dirac operator is
\begin{align}
D_F &:=e_{11}\otimes e_{11}\otimes
\begin{mfour}
0 & 0 & 0 & \Upsilon_\nu \\
0 & 0 & \Upsilon_\nu^t & \Upsilon_R \\
0 & \bar\Upsilon_\nu & 0 & 0 \\
\Upsilon_\nu^* & \Upsilon_R^* & 0 & 0
\end{mfour}
+e_{11}\otimes (1-e_{11})\otimes
\begin{mfour}
0 & 0 & 0 & \Upsilon_u \\
0 & 0 & \Upsilon_u^t & 0 \\
0 & \bar\Upsilon_u & 0 & 0 \\
\Upsilon_u^* & 0 & 0 & 0
\end{mfour}
\notag\\[5pt] & \quad
+e_{22}\otimes e_{11}\otimes
\begin{mfour}
0 & 0 & 0 & \Upsilon_e \\
0 & 0 & \Upsilon_e^t & 0 \\
0 & \bar\Upsilon_e & 0 & 0 \\
\Upsilon_e^* & 0 & 0 & 0
\end{mfour}
+e_{22}\otimes (1-e_{11})\otimes
\begin{mfour}
0 & 0 & 0 & \Upsilon_d \\
0 & 0 & \Upsilon_d^t & 0 \\
0 & \bar\Upsilon_d & 0 & 0 \\
\Upsilon_d^* & 0 & 0 & 0
\end{mfour} \;,
\label{eq:Dirac}
\end{align}
where each of the $\Upsilon$ matrices are in $M_n(\C)$, $\bar{m}:=(m^*)^t$ is the matrix
obtained from $m$ by conjugating each entry, and
we identify $\B(H_F)=
M_2(\C)\otimes M_4(\C)\otimes \big(M_4(\C)\otimes M_n(\C)\big)$
with $M_2(\C)\otimes M_4(\C)\otimes M_{4n}(\C)$ by writing
$M_{4n}(\C)$ as a $4\times 4$ matrix with entries in $M_n(\C)$; in particular
$e_{ij}\otimes m\in M_4(\C)\otimes M_n(\C)$ will be the matrix
with the $n\times n$ block $m$ in position $(i,j)$.

The matrix $\Upsilon_R$ is symmetric, the other $\Upsilon$ matrices are positive.
Their physical meaning is explained in section 17.4 of \cite{CM08}:
for $x=e,u,d$ the eigenvalues of $\Upsilon_x^*\Upsilon_x$ give the square of
the masses of the $n$ generations of the particle $x$;
the eigenvalues of $\Upsilon^*_\nu\Upsilon_\nu$ give the Dirac masses of neutrinos;
the eigenvalues of $\Upsilon^*_R\Upsilon_R$ give the Majorana masses of neutrinos.

If we replace a spectral triple with one that is unitary equivalent we do not change the
symmetries. From Theorem 1.187(3) (and analogously to Lemma 1.190) of \cite{CM08}
it follows that, modulo an unitary equivalence, we can diagonalize one element of each pair
$(\Upsilon_\nu,\Upsilon_e)$ and $(\Upsilon_u,\Upsilon_d)$. We choose to diagonalize
$\Upsilon_u$ and $\Upsilon_e$. 

Thus, we make the following hypothesis on the $\Upsilon$ matrices:
\begin{itemize}
\item $\Upsilon_u$ and $\Upsilon_e$ are positive, diagonal
and their eigenvalues are non-zero.
\item
$\Upsilon_d$ and $\Upsilon_\nu$ are positive,
the eigenvalues of $\Upsilon_d$ are non-zero.
Let us denote by $C$ the $SU(n)$ matrices such that 
$\Upsilon_d=C \delta_{\downarrow}C^*$, where $\delta_{\downarrow}$
is a diagonal matrix with non-negative eigenvalues.
$C$ is the so-called Cabibbo-Kobayashi-Maskawa matrix,
responsible for the quark mixing, cf.~Sec.~9.3 of \cite{CM08}.
Similarly the unitary diagonalizing $\Upsilon_\nu$ is the so-called Pontecorvo-Maki-Nakagawa-Sakata matrix,
responsible for the neutrino mixing, cf.~Sec.~9.6 of \cite{CM08}.
\item $\Upsilon_R$ is symmetric.
\item
For physical reasons, we assume that: $\Upsilon_x$ and
$\Upsilon_y$ have distinct eigenvalues, for all $x,y\in\{\nu,e,u,d\}$
with $x\neq y$; eigenvalues of $\Upsilon_e$, $\Upsilon_u$ and $\Upsilon_d$
are non-zero and with multiplicity one.
\end{itemize}

\begin{rem} \label{Upsilon_positive}
We will often use the fact that $  \Upsilon^t= \bar{ \Upsilon} $
for any positive matrix $\Upsilon$.
\end{rem}


\subsection{Quantum isometries of $F$}\label{sec:iso}

Since the definition of quantum isometry group is given for spectral triples over complex $\ast$-algebras, we first need
to explain how to canonically associate one to any spectral triple over a real $\ast$-algebra.

\begin{lemma}\label{lemma:spectral}
To any real spectral triple $(\A,\HH,D,\gamma,J)$ over a real $*$-algebra $\A$
we can associate a real spectral triple $(\B,\HH,D,\gamma,J)$ over
the complex $*$-algebra $\B\simeq\A_{\C}/\ker\pi_{\C}$, where
$\A_{\C}\simeq\A\otimes_{\R}\C$ is the complexification of $\A$,
with conjugation defined by $ ( a \otimes_{\R} z )^* = a^* \otimes_{\R} \bar{z}  $ for $ a \in \A $ and $ z \in \C$,
and $\pi_{\C}:\A_{\C}\to\B(\HH)$ is the $*$-representation
\begin{equation}\label{eq:piC}
\pi_{\C}(a\otimes_{\R}z)=z\pi(a) \;,\qquad a\in\A\,,\;z\in\C\,.
\end{equation}
\end{lemma}

\noindent
Notice that $\,\ker\pi_{\C} $ may be nontrivial since the representation $\pi_{\C}$ is not always faithful. For example, if
$\A$ is itself a complex $*$-algebra (every complex $*$-algebra is also a
real $*$-algebra) and $\pi$ is complex linear, then for any $a\in\A$ the element
$\,a\otimes_{\R}1+ia\otimes_{\R}i\,$ of $\A_{\C}$
is in the kernel of $\pi_{\C}$.
This happens in the Standard Model case, where the complexification of $A_F=\C\oplus\Q\oplus M_3(\C)$
is the algebra $(A_F)_{\C}:=\C\oplus\C\oplus M_2(\C)\oplus M_3(\C)\oplus M_3(\C)$,
where we have used the complex
$*$-algebra isomorphism $M_n(\C)\otimes_{\R}\C\to M_n(\C)\oplus M_n(\C)$
given by
$$
m\otimes_{\R}z \mapsto (mz,\bar m z)
$$
having inverse
\begin{equation}\label{eq:isoinv}
(m,m')\mapsto
\tfrac{m+\bar m'}{2}\otimes_{\R}1
+\tfrac{m-\bar m'}{2i}\otimes_{\R}i
\end{equation}
for all $m,m'\in M_n(\C)$, $z\in\C$.

Using \eqref{eq:piC}, \eqref{eq:isoinv} and \eqref{eq:repPi} we
get $\pi_{\C}(\lambda,\lambda',q,m,m')=\inner{\lambda,\lambda',q,m}$,
where
\begin{align}
\inner{\lambda,\lambda',q,m} &:=
q\otimes 1\otimes e_{11}\otimes 1 +
{\small\begin{mat}
\lambda & 0 \\ 0 & \lambda'
\end{mat}}\otimes 1
\otimes e_{44}\otimes 1 \notag\\[2pt]  & \qquad +
1\otimes
{\small
\begin{pmatrix}
\lambda & 0 & 0 & 0 \\
0 \\
0 && {\normalsize m} \\
0
\end{pmatrix}}
\otimes (e_{22}+e_{33})\otimes 1 \;.
\label{eq:B}
\end{align}
The complex $\ast$-algebra $B_F:=(A_F)_{\C}/\ker\pi_{\C}$ is simply
the algebra $B_F\simeq \C\oplus\C\oplus M_2(\C)\oplus M_3(\C)$
with elements $\inner{\lambda,\lambda',q,m}$.
With $A_F$ replaced by $B_F,$ we can now study quantum isometries.

We notice that in the case of the spectral triple of the internal part of the Standard Model, the conditions (\ref{eq:cond2}-\ref{eq:cond3}) are equivalent to
\begin{subequations}
\begin{align}
& \hspace{-3.5cm}
(J_0\otimes 1)\bar U=U(J_0\otimes 1) \;; \label{eq:cond5} \\[4pt]
& \hspace{-3.5cm}
\mathrm{Ad_U}(B_F)\subset B_F\otimes_{\mathrm{alg}}Q\;; \label{eq:cond6}
\end{align}
\end{subequations}
with $J_0$ given by \eqref{eq:J0}.
The equivalence between \eqref{eq:cond2} and \eqref{eq:cond5} is an immediate
consequence of the definition of $J_F$.
The equivalence between \eqref{eq:cond3} and \eqref{eq:cond6} follows from the equality of
$B_F''$ and $B_F,$ since the latter is a finite-dimensional $C^*$-algebra.

We need a preparatory lemma before our main proposition.

\begin{lemma}\label{prop:mainA}
Let $Q$ be the universal $C^*$-algebra generated by
unitary elements $x_k$ ($k=0,\ldots,n$),
the matrix entries of $3\times 3$ biunitaries
$T_m$ ($m=1,\ldots,n$) and of an $n\times n$
biunitary $V$, with relations
\begin{subequations}\label{eq:qisot}
\begin{gather}
\mathrm{diag}(x_0x_1,...,x_0x_n)\Upsilon_\nu =\Upsilon_\nu \mathrm{diag}(x_0x_1,...,x_0x_n)=
\bar V\Upsilon_\nu=\Upsilon_\nu\bar V\,,
 \qquad
V\Upsilon_R =\Upsilon_R\hspace{1pt}\bar V \,, \label{eq:qisotA} \\[4pt]
\sum\nolimits_{m=1}^nC_{rm}\bar C_{sm}(T_m)_{j,k}=0 \,,\qquad\forall\;r\neq s\,,\quad (\,r,s=1,\ldots,n;\,j,k=1,2,3\,) \label{eq:qisotC} \\[4pt]
(T_m^*)_{i,j}(T_m)_{k,l}=(T_{m'}^*)_{i,j}(T_{m'})_{k,l} \,,\qquad\forall\;m,m'\,,\quad (\,i,j,k,l=1,2,3,\; m,m'=1,\ldots,n\,)
\label{eq:qisotB}
\end{gather}
\end{subequations}
where $C=((C_{r,s}))$ is the CKM matrix.
Then $Q$ with matrix coproduct
\begin{equation}\label{eq:zzz}
\Delta(x_k)=x_k\otimes x_k \,,\quad
\Delta((T_m)_{ij})=\sum_{l=1,2,3}
(T_m)_{il}\otimes (T_m)_{lj} \,,\quad
\Delta(V_{ij})=\sum_{l=1,\ldots,n}
V_{il}\otimes V_{lj} \,,
\end{equation}
is a quantum subgroup of the free product
\begin{equation}\label{eq:qsubG}
\underbrace{C(U(1))*C(U(1))*\ldots*C(U(1))}_{n+1}
\,*\,\,Q_n(3) *A_u(n) \;.
\end{equation}
The Woronowicz $C^*$-ideal of \eqref{eq:qsubG} defining $Q$
is determined by the relations \eqref{eq:qisotA} and \eqref{eq:qisotC}.
\end{lemma}
\begin{proof}
$Q_n(3)$ is by definition generated by $3\times 3$ biunitaries $T'_m$ ($m=1,\ldots,n$)
with the relation \eqref{eq:qisotB},
$A_u(n)$ is generated by the matrix entries of a $n\times n$ biunitary $V'$, and 
$C(U(1))*C(U(1))*\ldots*C(U(1))$ is freely generated by unitary
elements $x'_k$ ($k=0,\ldots,n$).
The map $T'_m\mapsto T_m$, $V'\mapsto V$ and $x_k'\mapsto x_k$ defines a surjective
$C^*$-algebra morphism from the CQG in \eqref{eq:qsubG} to $Q$.

From Lemma \ref{lemma:6.10}, it follows that the kernel of the morphism $(V',x_k')\mapsto (V,x_k)$
is a Woronowicz $C^*$-ideal, i.e.~the relations \eqref{eq:qisotA} define a quantum
subgroup of $C(U(1))*C(U(1))*\ldots*C(U(1))*A_u(n)$ (apply the Lemma to $A=\Upsilon_\nu$ and
$X,Y\in\{\mathrm{diag}(x_0x_1,...,x_0x_n),V\}$).

It remains to prove that the kernel $I$ of the morphism $T_m'\mapsto T_m$ is also a
Woronowicz $C^*$-ideal, i.e.~the quotient of $Q_n(3)$ by the relation \eqref{eq:qisotC}
is a CQG. The ideal $I$ is generated by the elements 
$X_{r,s,j,k}:=\sum\nolimits_{m=1}^nC_{rm}\bar C_{sm}(T'_m)_{j,k}$ for
all $j,k=1,2,3$, $r,s=1,\ldots,n$ and $r\neq s$. An easy computation shows that
\begin{align*}
\Delta(X_{r,s,j,k}) &=
\sum\nolimits_{m=1}^n\sum\nolimits_{l=1}^3C_{rm}\bar C_{sm}(T'_m)_{j,l}\otimes (T'_m)_{l,k} \\
&=
\sum\nolimits_{l,p=1}^3
\sum\nolimits_{m=1}^nC_{rm}\bar C_{pm}(T'_m)_{j,l}
\otimes 
\sum\nolimits_{m'=1}^nC_{pm'}\bar C_{sm'}(T'_{m'})_{l,k} \\
&=
\sum\nolimits_{l,p=1}^3 X_{r,p,j,l}\otimes X_{p,s,l,k} \;,
\end{align*}
where the second equality follows from
$\sum_{p=1}^3\bar C_{pm}C_{pm'}=(C^*C)_{mm'}=\delta_{mm'}$ (recall
that $C$ is a unitary matrix).
Hence $\Delta(I)\subset I\otimes I$, so that $I$ is a Woronowicz $C^*$-ideal.
This concludes the proof.
\end{proof}

\begin{prop}\label{prop:main}
The universal object $\qisot(D_F)$ of the category $\catA$ is given by the CQG in 
Lemma~\ref{prop:mainA} with corepresentation
\begin{align}
U\;= & \;\;
e_{11}\otimes e_{11}\otimes e_{11}\otimes \sum^n_{k = 1} e_{kk}\otimes x_0 x_k
+e_{22}\otimes e_{11}\otimes (e_{11}+e_{44})\otimes \sum^n_{k = 1} e_{kk}\otimes x_k
\notag\\[2pt]
& +
e_{11}\otimes e_{11}\otimes e_{33} \otimes \sum^n_{k = 1} e_{kk}\otimes x^*_k x^*_0
+e_{22}\otimes e_{11}\otimes (e_{22}+e_{33})\otimes \sum^n_{k = 1} e_{kk}\otimes x^*_k
\notag\\[2pt]
& +
e_{11}\otimes e_{11}\otimes e_{22}\otimes\sum^n_{j,k = 1}e_{jk}\otimes ( V )_{jk}
+e_{11}\otimes e_{11}\otimes e_{44} \otimes \sum^n_{j,k = 1} e_{jk}\otimes (\bar{V} )_{jk}
\notag\\[2pt]
& +
e_{11}\otimes \sum_{j,k = 1,2,3} e_{j+1,k+1}\otimes (e_{11}+e_{44}) \otimes \sum^n_{m = 1} e_{mm}\otimes  (T_m)_{j,k}
\notag\\[2pt]
& +
e_{22}\otimes \sum_{j,k = 1,2,3} e_{j+1,k+1}\otimes (e_{11}+e_{44}) \otimes \sum^n_{m = 1} e_{mm}\otimes x^*_0 (T_m)_{j,k}
\notag\\[2pt]
& +
e_{11}\otimes \sum_{j,k = 1,2,3} e_{j+1,k+1}\otimes (e_{22}+e_{33}) \otimes \sum^n_{m = 1} e_{mm}\otimes (\bar T_m)_{j,k}
\notag\\[2pt]
& +
e_{22}\otimes \sum_{j,k = 1,2,3} e_{j+1,k+1}\otimes (e_{22}+e_{33}) \otimes \sum^n_{m = 1} e_{mm}\otimes (\bar T_m)_{j,k} x_0 \;.
\label{eq:Ufourth}
\end{align}
$\qisot(D_F)$ coacts trivially on the two summands $\C$ of
$B_F=\C\oplus\C\oplus M_2(\C)\oplus M_3(\C)$, while
on the remaining summands the coaction is
\begin{subequations}\label{eq:Ucoact}
\begin{align}
\alpha(\inner{0,0,e_{ii},0}) &= \inner{0,0,e_{ii},0}\otimes 1 \;, \\
\alpha(\inner{0,0,e_{12},0}) &= \inner{0,0,e_{12},0}\otimes x_0 \;, \\
\alpha(\inner{0,0,e_{21},0}) &= \inner{0,0,e_{21},0}\otimes x_0^* \;, \\
\alpha(\inner{0,0,0,e_{ij}}) &= \sum\nolimits_{k,l=1,2,3}\inner{0,0,0,e_{kl}}\otimes
  ( T_1^*)_{ i,k} (T_1)_{ l,j} \;.
\end{align}
\end{subequations}
\end{prop}

\begin{proof}
The proof is in Sec.~\ref{sec:7}.
\end{proof}

\begin{df}
Let $Q_{n,C}(3)$ be the quantum subgroup of $Q_n(3)$, cf.~Def.~\ref{def:amalg},
defined by the relation $\sum\nolimits_{m=1}^nC_{rm}\bar C_{sm}(u_m)_{j,k}=0$.
\end{df}

\begin{rem}\label{rem:QnC}
It is easy to see that $Q_{n,C}(3)$ is noncommutative as a $C^*$-algebra.
Indeed, if $u$ is a $3\times 3$ biunitary generating $A_u(3)$,
the map
$$
(u_m)_{jk}\mapsto u_{jk}\,,\;\forall\;m=1,\ldots,n\,,\;j,k=1,2,3\,,
$$
is a $C^*$-algebra morphism
($C$ is a unitary matrix, hence \eqref{eq:qisotC}
and \eqref{eq:qisotB} are automatically satisfied). 
Thus $ A_u ( 3 ) $ is a quantum subgroup of $ Q_{n,C}(3)$.
\end{rem}

\begin{prop}
The quantum isometry group of the internal space of the Standard
Model is
$$
\qiso ( D_F ) =C(U(1))*A_{\mathrm{aut}}(M_3(\C)) \;.
$$
Its abelianization is given by (complex functions on) the classical
group $U(1)\times PU(3)$.
\end{prop}

\begin{proof}
From \eqref{eq:Ucoact} it follows that $\qiso ( D_F ) $ is generated by
$x_0$ and $(T_1^*)_{ i,k}(T_1)_{ l,j}$: then $\qiso ( D_F )$ is a quantum
subgroup of $C(U(1))*PA_u(3)$.
On the other hand $A_u(3)$ is a quantum subgroup of $Q_{n,C}(3)$
(Rem.~\ref{rem:QnC}), and with the map $x_0\mapsto x_0$,
$x_i\mapsto 1$ ($i=1,\ldots,n$), $V\mapsto x_0 1_n$ and
$(T_m)_{jk}\mapsto u_{jk}\;\forall\;m=1,\ldots,n$ (with $u_{jk}$
the usual generators of $A_u(3)$), one proves that 
$C(U(1))*A_u(3)$ is a sub-object of $\qisot(D_F)$
in the category $\catA$ and
$C(U(1))*PA_u(3)$ is a quantum subgroup of $\qiso (D_F)$;
hence $\qiso ( D_F )$ and $C(U(1))*PA_u(3)$ coincide.
Recalling that $PA_u(3)\simeq A_{\mathrm{aut}}(M_3(\C)) $
(cf.~Def.~\ref{def:Ban} and Prop.~\ref{thm:Ban})
the proof is concluded.
\end{proof}

Although $\qisot ( D_F ) $ depends on $\Upsilon_\nu$ , $\Upsilon_R$ and the CKM matrix $C$
(cf.~\eqref{eq:qisot}), the quantum group $\qiso ( D_F )$ does not depend on the explicit form
of these two matrices. We stress the importance of this results,
since neutrino masses are not known (at the moment, we only
know that they are all distinct \cite{Fuk98,Ahm02}).
Also, $\qiso ( D_F )$ is independent on the number of generations.

Let us conclude this section by explaining how elementary particles transform
under the corepresentation $U$ in physics notation.
As explained in Sec.~\ref{sec:SM}, we have
\begin{align*}
\nu_{L,k} &:=e_1\otimes e_1\otimes e_1\otimes e_k \;,
&& \text{(left-handed neutrino, generation $k$)}
\\
\nu_{R,k} &:=e_1\otimes e_1\otimes e_4\otimes e_k \;,
&& \text{(right-handed neutrino, generation $k$)}
\\
e_{L,k} &:=e_2\otimes e_1\otimes e_1\otimes e_k \;,
&& \text{(left-handed electron, generation $k$)}
\\
e_{R,k} &:=e_2\otimes e_1\otimes e_4\otimes e_k \;,
&& \text{(right-handed electron, generation $k$)}
\\
u_{L,c,k} &:=e_1\otimes e_{c+1}\otimes e_1\otimes e_k \;,
&& \text{(left-handed up-quark, color $c$, generation $k$)}
\\
u_{R,c,k} &:=e_1\otimes e_{c+1}\otimes e_4\otimes e_k \;,
&& \text{(right-handed up-quark, color $c$, generation $k$)}
\\
d_{L,c,k} &:=e_2\otimes e_{c+1}\otimes e_1\otimes e_k \;,
&& \text{(left-handed down-quark, color $c$, generation $k$)}
\\
d_{R,c,k} &:=e_2\otimes e_{c+1}\otimes e_4\otimes e_k \;,
&& \text{(righ-handed down-quark, color $c$, generation $k$)}
\end{align*}
where $\{e_i\,,\,i=1,\ldots,r\}$ is the canonical orthonormal basis of $\C^r$,
$c=1,2,3$ and $k=1,\ldots,n$.
These together with the corresponding antiparticles form a linear basis
of $H_F$. A straightforward computation using \eqref{eq:Ufourth} proves
that we have the following transformation laws
\begin{align*}
U(\nu_{L,k}) &:=\nu_{L,k}\otimes x_0x_k \;,
&
U(\nu_{R,k}) &:=\sum\nolimits_{j=1}^n\nu_{R,j}\otimes\bar V_{jk} \;,
\\
U(e_{L,k}) &:=e_{L,k}\otimes x_k \;,
&
U(e_{R,k}) &:=e_{R,k}\otimes x_k \;,
\\[3pt]
U(u_{L,c,k}) &:=\sum\nolimits_{c'=1}^3u_{L,c',k}\otimes (T_k)_{c'c} \;,
&
U(u_{R,c,k}) &:=\sum\nolimits_{c'=1}^3u_{R,c',k}\otimes (T_k)_{c'c} \;,
\\
U(d_{L,c,k}) &:=\sum\nolimits_{c'=1}^3d_{L,c',k}\otimes x_0^*(T_k)_{c'c} \;,
&
U(d_{R,c,k}) &:=\sum\nolimits_{c'=1}^3d_{R,c',k}\otimes x_0^*(T_k)_{c'c} \;,
\end{align*}
where $U(v)$, $v\in H_F$, is a shorthand notation for $U(v\otimes 1_Q)$.
Antiparticles transform according to the conjugate corepresentations.

We comment now on the meaning of $\qisot(D_F)$.
\begin{rem}\label{rem:gauge}
Let $z$ be the generator of $C(U(1))$, $T=((T_{jk}))$ be the generators of $A_u(3)$
and consider the corepresentation $H_F\to H_F\otimes (C(U(1))*A_u(3))$
determined by 
\begin{subequations}\label{eq:symbreak}
\begin{align}
\nu_{\bullet,k} &\mapsto \nu_{\bullet,k}\otimes 1 \;,
&
e_{\bullet,k} &\mapsto e_{\bullet,k}\otimes (z^*)^3 \;,\\
u_{\bullet,c,k} &\mapsto 
\sum\nolimits_{c'=1}^3u_{\bullet,c',k}\otimes z^2T_{c'c} \;,
&
d_{\bullet,c,k} &\mapsto
\sum\nolimits_{c'=1}^3d_{\bullet,c',k}\otimes z^*T_{c'c} \;,
\end{align}
\end{subequations}
where $\bullet$ is $L$ or $R$. Let $q$ be a third root of unity,
and consider the $\Z_3$ action on $C(U(1))*A_u(3)$ given by $z\mapsto qz$,
$T_{jk}\mapsto qT_{jk}$. 
The elements appering in the image of the above corepresentation
generate the fixed point subalgebra for this action, that is 
$\{C(U(1))*A_u(3)\}^{\Z_3}$.

The quantum group $\{C(U(1))*A_u(3)\}^{\Z_3}$ with the corepresentation
above is a sub-object of $\qisot(D_F)$ in the category $\catA$.
The surjective CQG homorphism $\qisot( D_F )\to\{C(U(1))*A_u(3)\}^{\Z_3}$
is given by
$$
x_0\mapsto z^3 ,\quad
x_m\mapsto (z^*)^3 ,\;\;\forall m=1,\ldots,n\,,\quad
(T_m)_{jk}\mapsto z^2T_{jk},\;\;\forall m=1,\ldots,n\,,\quad
V\mapsto 1_n \,.
$$
The kernel of this map ---
the ideal generated by $V_{jk}$ and by products $x_0x_k$
and $(T_m^*T_{m'})_{kl}$ for all $m\neq m'$
--- is given by elements that do not appear in the adjoint
coaction on $B_F$. Roughly speaking, modulo terms ``commuting''
with the algebra $B_F$, we have that $\qisot( D_F )\sim\{C(U(1))*A_u(3)\}^{\Z_3}$
is the ``free version'' of the ordinary gauge group
after symmetry breaking.

If we pass to the abelianization $C(U(1)\times U(3))^{\Z_3}\simeq
C\big(\bigl(U(1)\times U(3)\bigr)/\Z_3\big)$ of $\{C(U(1))*A_u(3)\}^{\Z_3}$
and from the corresponding corepresentation to the dual representation
of $(\tau,g)\in U(1)\times U(3)$,
from \eqref{eq:symbreak} we find the usual global gauge transformations
after symmetry breaking:
$$
\nu_{\bullet,k}\mapsto \nu_{\bullet,k}, \,\quad
e_{\bullet,k} \mapsto (\tau^*)^3e_{\bullet,k}, \,\quad
u_{\bullet,c,k}\mapsto \sum\nolimits_{c'=1}^3\tau^2 g_{c'c}u_{\bullet,c',k}, \,\quad
d_{\bullet,c,k}\mapsto\sum\nolimits_{c'=1}^3 \tau^* g_{c'c}d_{\bullet,c',k}.
$$
\end{rem}

\subsection[Special cases]{$\widetilde{\text{QISO}}^+_J$ in two special cases}\label{minimalmodel}

As we already noticed, $\qisot ( D_F ) $ depends upon the explicit form of
$\Upsilon_\nu$, $\Upsilon_R$ and $C$. In particular, on one extreme we
have the case when $\Upsilon_\nu$ is invertible
(this is the case of the Dirac operator in the moduli space as in Prop.~1.192 of
\cite{CM08}) and on the other extreme we have the case $\Upsilon_\nu=0$.

\begin{prop}
If $ \Upsilon_\nu $ is invertible, $\qisot ( D_F ) $ is the free product
of $Q_{n,C}(3)$ with the quotient of
$$
\underbrace{C(U(1))*C(U(1))*\ldots*C(U(1))}_{n+1}
$$
by the relations
\begin{align*}
x^*_i x^*_0 &= x_0 x_j \qquad\forall\;i,j\;\text{such that}\;(\Upsilon_R)_{ij}\neq 0\;,\\
x_i & = x_j \qquad\forall\;i,j\;\text{such that}\;(\Upsilon_\nu)_{ij}\neq 0\;.
\end{align*}
\end{prop}

\begin{proof}
If $\Upsilon_\nu$ is invertible, the first equation in
\eqref{eq:qisotA} gives $V={\rm diag} ( x^*_1 x^*_0,\ldots, x^*_n x^*_0 )$
(so that the factor $A_u(n)$ in \eqref{eq:qsubG} disappears) and
also $(\Upsilon_\nu)_{ij}x_0(x_i - x_j)=0$.
The latter implies $x_i=x_j$ whenever $(\Upsilon_\nu)_{ij}\neq 0$.

The second equation in \eqref{eq:qisotA} becomes $(\Upsilon_R)_{ij}(x^*_i x^*_0 - x_0 x_j)=0$,
which implies $x^*_i x^*_0 = x_0 x_j$ whenever $(\Upsilon_R)_{ij}\neq 0$.
\end{proof}

Although disproved by experiment, it is an interesting exercise
to study the case of massless ($\Upsilon_\nu=0$) left-handed
neutrinos, that is the so-called \emph{minimal} Standard Model.

\begin{prop}\label{prop:min}
If $\Upsilon_\nu=0$, $\qisot ( D_F )$ is isomorphic to
$$
\underbrace{C(U(1))*C(U(1))*\ldots*C(U(1))}_{n + 1 }\,\,*\,\,Q_{n,C}(3) * A',
$$
where $A':= \,A_u(n)/\!\sim\;$, $A_u(n)$ is generated by the $n\times n$
biunitary $V$ and ``$\sim$'' is the relation $V\Upsilon_R =\Upsilon_R\bar{ V }$.
\end{prop}

\smallskip

As a consequence of Noether's theorem, any Lie group symmetry is associated
to a corresponding conservation law. 
We shall see in Sec.~\ref{sec:4.2} that $\qisot(D_F)$ is indeed a symmetry of the dynamics.
In Rem.~\ref{rem:gauge}, roughtly speaking,
we discussed the part of $\qisot$ that is relevant in the coaction on the algebra
$B_F$: it is the free version of the gauge group which corresponds to the conservation
of color and electric charge.
We complete here the analysis by discussing the additional symmetries that are present
in the case of the minimal Standard Model in Prop.~\ref{prop:min}.

The factor $A'$ coacts only on the subspace
$(e_{11}\otimes e_{11}\otimes (e_{22}+e_{44})\otimes 1)H_F$
of right-handed neutrinos, and can be neglected in the
minimal Standard Model (where we consider only left-handed
neutrinos).
As a consequence of Noether's theorem, there exists a conservation law corresponding to each classical group
of symmetries.

It is easy to give an interpretation to the $C(U(1))$
factors generated by $x_i$, $i=1,\ldots,n$. Passing from the $C(U(1))$
coaction to the dual $U(1)$ action, one easily sees that for $ i > 0, $
$x_i$ gives a phase transformation of the $i$-th
generation of $\nu_L,e_L,e_R$ (plus the opposite transformation
for the antiparticles).
In the minimal Standard Model, which has only
left-handed (massless) neutrinos, these symmetries give
the conservation laws of the total number of leptons in
each generation (electron number, muon number, tau number,
plus other $n-3$ for the other families of leptons).

To conclude the list of conservation laws, there is still
one classical $U(1)$ subgroup of the factor $Q_{n,C}(3)$ that
should be mentioned.
If we denote by $y$ the unitary generator of
$C(U(1))$, a surjective CQG homomorphism $\varphi:\qisot ( D_F )\to C(U(1))$
is given by
$$
x_0 \mapsto 1\;, \qquad
x_i \mapsto 1\;, \qquad
V_{j,k} \mapsto \delta_{j,k} \;, \qquad
(T_i)_{j,k} \mapsto\delta_{j,k} y\;,
$$
for all $i=1,\ldots,n$ and $j,k=1,2,3$.
From $U$ we get the following corepresentation of this $U(1)$ subgroup
on $H_F$:
\begin{align*}
(id\otimes\varphi)(U) &=
1\otimes e_{11}\otimes 1\otimes 1\otimes 1_{C(U(1))}
\\[2pt] & \quad +
1\otimes (1-e_{11})\otimes (e_{11}+e_{44})\otimes 1\otimes y
\\[2pt] & \quad +
1\otimes (1-e_{11})\otimes (e_{22}+e_{33})\otimes 1\otimes y^*
\;.
\end{align*}
The representation of $U(1)$ dual to this corepresentation of $C(U(1))$
is given by a phase transformation on the subspace
$\C^2\otimes(1-e_{11})\C^4\otimes (e_{11}+e_{44})\C^4\otimes\C^n$
of quarks and the inverse transformation on the subspace
$\C^2\otimes(1-e_{11})\C^4\otimes (e_{22}+e_{33})\C^4\otimes\C^n$
of anti-quarks and is called in physics the ``baryon phase symmetry''.
It corresponds to the conservation of the baryon number (total number
of quarks minus the number of anti-quarks).

\medskip

In this section we discussed conservation laws associated to
classical subgroups of $\qisot ( D_F )$ in the massless
neutrino case.
It would be interesting to extend this study to the full quantum
group $\qisot ( D_F )$ in the sense of a suitable Noether analysis extended to the
quantum group framework.
If we consider massive neutrinos, we lose a lot of
classical symmetries, but we still have many quantum symmetries.
A natural question is whether quantum symmetries are suitable
for deriving conservation laws (i.e.~physical predictions).
A first step in this direction is to investigate whether the spectral action
is invariant under quantum isometries. We discuss this point in the next
section.


\section{Quantum isometries of $M\times F$}\label{sec:5}
\subsection{Quantum isometries of a product of spectral triples} \label{sec:5a}
Before discussing the spectral action, we want to understand whether the quantum isometry group of the finite geometry $F$ is also a quantum group
of orientation preserving isometries of the full spectral triple of the Standard Model, that is the product of $F$ with the canonical spectral
triple of a compact Riemannian spin manifold $M$ with no boundary. The answer is affirmative and we can prove it in a more
general situation:
\begin{itemize}
\punto
Let $(\A_1,\HH_1,D_1,\gamma_1,J_1)$ be any unital real spectral triple ($\gamma_1=1$ if the spectral triple is odd).

\punto
Let $(\A_2,\HH_2,D_2,\gamma_2,J_2)$ be a finite-dimensional unital even real spectral triple.

\punto
Let $(\A,\HH,D,\gamma,J)$ be the product triple, i.e.
\begin{gather*}
\A:=\A_1\otimes_{\mathrm{alg}}\A_2 \;, \qquad
\HH:=\HH_1\otimes\HH_2 \;, \qquad
D:=D_1\otimes\gamma_2+1\otimes D_2 \;,\\[2pt]
\gamma:=\gamma_1\otimes \gamma_2 \;, \qquad
J:=J_1\otimes J_2 \;.
\end{gather*}
\end{itemize}
In the case of the Standard Model, $(\A_1,\HH_1,D_1,\gamma_1,J_1)$ and $(\A_2,\HH_2,D_2,\gamma_2,J_2)$ will be the canonical spectral triple of $M$ and the spectral triple $(B_F,H_F,D_F,\gamma_F,J_F)$ respectively.

We claim that:

\begin{lemma}\label{lemma:product}
 $\widetilde{\rule{0pt}{7pt}\smash[t]{\mathrm{QISO}}}\rule{0pt}{8pt}^+(\A_2,\HH_2,D_2,\gamma_2,J_2)$ coacts by ``orientation and
real structure preserving isometries'' on the product triple $(\A,\HH,D,\gamma,J)$.
\end{lemma}

\begin{proof}
Let $Q_0$ be the quantum group $\widetilde{\rule{0pt}{7pt}\smash[t]{\mathrm{QISO}}}\rule{0pt}{8pt}^+_{J_2}( D_2 )$ and $U$ its corepresentation on $\HH_2$.
Then $\hat{U}:=1\otimes U$ is a unitary corepresentation on $\HH_1\otimes\HH_2$, and we need to prove that
it satisfies \eqref{eq:cond1}, \eqref{eq:cond2}, and \eqref{eq:cond3}.
The first two conditions are easy to check. Indeed, if $U$ commutes with $D_2$ and $\gamma_2$, clearly $1\otimes U$
commutes with $D=D_1\otimes\gamma_2+1\otimes D_2$ and $\gamma=\gamma_1\otimes \gamma_2$.
Moreover, for any vector $\xi=\xi_1\otimes\xi_2\in\HH_1\otimes\HH_2$,
\begin{align*}
(J\otimes *)\hat{U}(\xi\otimes 1) &=
(J_1\otimes J_2\otimes *)(1\otimes U)(\xi_1\otimes\xi_2\otimes 1) \\
&=J_1\xi_1\otimes (J_2\otimes *)U(\xi_2\otimes 1) \\
&=J_1\xi_1\otimes U(J_2\xi_2\otimes 1) \\
&=(1\otimes U)(J_1\xi_1\otimes J_2\xi_2\otimes 1) \\
&=\hat{U}(J\xi\otimes 1) \;,
\end{align*}
and thus \eqref{eq:cond2} is proved.

Any element of $\A$ is a finite sum of tensors $a_1\otimes a_2$, with $a_1\in\A_1$ and
$a_2\in\A_2$, and since $\A_2$ is finite dimensional implies  $U(a_2\otimes 1_{Q_0})U^*\in\A_2\otimes_{\mathrm{alg}} Q_0$, we have
$$
\mathrm{Ad_{\hat{U}}}  (a_1\otimes a_2)=\hat{U}(a_1\otimes a_2\otimes 1_{Q_0})\hat{U}^*=
a_1\otimes U(a_2\otimes 1_{Q_0})U^*\in \A_1\otimes_{\mathrm{alg}}\A_2\otimes_{\mathrm{alg}}Q_0
$$
which implies \eqref{eq:cond3}.
\end{proof}

\subsection{Invariance of the spectral action}\label{sec:4.2}
The dynamics of a unital spectral triple $(\A,\HH,D,\gamma)$
--- with $\gamma=1$ in the odd case ---
is governed by an action functional~\cite{CC97}
$$
S[A,\psi]:=S_b[A]+S_f[A,\psi] \;,
$$
whose variables are 
a self-adjoint one-form $A\in \Omega^{1,s.a.}_D\subset\B(\HH)$
and $\psi$ either in $\HH$ or in $\HH_+:=(1+\gamma)\HH$.
While one uses $\HH$ in Yang-Mills theories, the reduction to $\HH_+$ is employed in
the Standard Model to solve the fermion doubling problem \cite{LMMS96,CM08}.
The fermionic part of the spectral action is either
\begin{equation}\label{eq:Sfa}
S_f[A,\psi]=\inner{\psi,D_A\psi} \;,
\qquad
D_A:=D+A \;,
\end{equation}
or for a real spectral triple
\begin{equation}\label{eq:Sfb}
S_f[A,\psi]:=\inner{J\psi,D_A\psi} \;,
\qquad
D_A:=D+A+\epsilon' JAJ^{-1} \;,
\end{equation}
where $\epsilon'$ is the sign in \eqref{eq:yyy}.
The bosonic part is
$$
S_b[A]=\tr\,f(D_A/\Lambda) \;,
$$
where $D_A$ is either the operator in \eqref{eq:Sfa} or \eqref{eq:Sfb},
and $f$ is a suitable cut-off function (with $\Lambda>0$).
More precisely, $f$ is a smooth approximation of the characteristic function of the interval $[-1,1]$,
so that $f(D_A/\Lambda)$ --- defined via the continuous functional calculus --- is a trace class operator on $\HH$ and $S_b[A]$ is well defined.

In the rest of the section we focus on the fermionic action $S_f$ and
the operator $D_A$ given by \eqref{eq:Sfb}, although all the proofs
can be repeated in the case \eqref{eq:Sfa} as well.

Assume that $Q$ is a CQG with a unitary corepresentation $\hat U$ on $\HH$ commuting with $D$
and $\gamma$, and such that $\mathrm{Ad_{\hat U}}$ maps $\A$ into $\A\otimes_{\text{alg}}Q$
(rather than \eqref{eq:cond3}).
Then $\HH_+$ is preserved by $\hat U$, and for any $1$-form
$A=\sum_ia_i[D,b_i]$, with $a_i,b_i\in\A$, the operator
$\mathrm{Ad_{\hat U}}(A)=\hat{U}(A\otimes 1){\hat{U}}^*=
\sum_i\mathrm{Ad_{\hat U}}(a_i)[D\otimes 1,\mathrm{Ad_{\hat U}}(b_i)]$ is an element
of $\Omega^1_D\otimes_{\text{alg}}Q$. Therefore a coaction of $Q$ on $\Omega^1_D\oplus\HH_+$ is given by
$$
\beta:(A,\psi)\mapsto\bigl(\hat{U}(A\otimes 1){\hat{U}}^*,\hat{U}(\psi \otimes 1)\bigr) \;.
$$
To discuss the (co)invariance of the spectral action we need to extend it to the latter space.
There is a natural way to do it. The inner product $\inner{\,,\,}:\HH_+\otimes\HH_+\to\C$ can be
extended in a unique way to an Hermitian structure $\inner{\,,\,}_Q:\mathcal{M}\otimes\mathcal{M}\to Q$
on the right $Q$-module $\mathcal{M}:=\HH_+\otimes Q$ by the rule $\inner{\psi\otimes q,\psi'\otimes q'}_Q=q^*q'\inner{\psi,\psi'}$.
Unitary (resp.~antiunitary) maps $L$ on $\HH_+$ are extended in a unique way to $Q$-linear
(resp.~antilinear) maps on $\mathcal{M}$ as $L\otimes 1$ (resp.~$L\otimes *$).
The corresponding extension of the spectral action is given by the $Q$-valued functional
\begin{align*}
\tilde S[\tilde A,\tilde \psi] &:=\tilde S_b[\tilde A]+\tilde S_f[\tilde A,\tilde \psi] \;,
\intertext{where}
\tilde S_b[\tilde A] &:= ( \tr_{\HH} \otimes {\rm id} ) \,f(D_{\hspace{-1pt}\tilde A}/\Lambda) \;,\\[4pt]
\tilde S_f[\tilde A,\tilde \psi] &:=\bigl<(J\otimes *)\tilde \psi,D_{\hspace{-1pt}\tilde A}\tilde \psi\bigr>_Q \;,
\end{align*}
and $\tilde A$ is a self-adjoint element of $\Omega^1_D\otimes_{\text{alg}}Q$, $\tilde \psi\in\HH_+\otimes Q$,
$D_{\hspace{-1pt}\tilde A}:=D\otimes 1+\tilde A+\epsilon' (J\otimes*)\tilde A(J\otimes *)^{-1}$.

Here $f(D_{\hspace{-1pt}\tilde A}/\Lambda)$ is defined in the following way:
if $L^2(Q)$ is the GNS representation associated to the Haar state of $Q$, then
$\tilde A+\epsilon' (J\otimes*)\tilde A(J\otimes *)^{-1}$ is a bounded self-adjoint
operator on $\HH\otimes L^2(Q)$ and $D_{\hspace{-1pt}\tilde A}$ is a (unbounded) self-adjoint operator on the Hilbert space $\HH\otimes L^2(Q)$. The operator $f(D_{\hspace{-1pt}\tilde A}/\Lambda)$ is then defined using the continuous functional calculus.

\smallskip

By (co)invariance of the action functional we mean the property
\begin{equation}\label{eq:inv}
\tilde S[\beta(A,\psi)]=S[A,\psi]\cdot 1_Q.
\end{equation}
Notice that 
since $A$ is a self-adjoint $1$-form,
$\tilde A=\hat{U}(A\otimes 1){\hat{U}}^*$ is 
a self-adjoint element of $\Omega^1_D\otimes_{\text{alg}}Q$ as required above
so that $\tilde S[\beta(A,\psi)]$ is well defined.
In the remaining part of the section we discuss the invariance of
the action. We study separately the fermionic and the bosonic part.

\begin{prop} \label{invariance-fermionic}
If $\, \hat U $ satisfies \eqref{eq:cond1} and \eqref{eq:cond2}, then
$$\tilde S_f[\beta(A,\psi)]=S_f[A,\psi]\cdot 1_Q$$
for all $(A,\psi)\in \Omega^{1,s.a.}_D\oplus\HH_+$.
\end{prop}

\begin{proof}
This is a simple algebraic identity. Since $\hat U$ commutes with $D$ and $J\otimes *$,
we have
\begin{equation}\label{eq:UD}
D_{\hat U(A\otimes 1)\hat U^*}=
D\otimes 1+\hat U(A\otimes 1)\hat U^*+\epsilon' (J\otimes*)\hat U(A\otimes 1)\hat U^*(J\otimes *)^{-1}
=\hat U(D_A\otimes 1)\hat U^*.
\end{equation}
Thus,
\begin{align*}
\tilde S_f[\beta(A,\psi)] &=\bigl<(J\otimes *)\hat
U(\psi\otimes 1_Q),D_{\hat U(A\otimes 1)\hat U^*}\hat U(\psi\otimes 1_Q)\bigr>_Q \\
&=\bigl<\hat U(J\psi\otimes 1_Q),\hat U(D_A\psi\otimes 1_Q)\bigr>_Q \\
&=\inner{J\psi,D_A\psi}\cdot 1_Q  =S_f[A,\psi]\cdot 1_Q \;,
\end{align*}
by the unitarity of $ \hat U. $
\end{proof}

For the rest of the subsection, we will assume that $(\A,\HH,D,J,\gamma)$ is the product of two real spectral triples, one of them being even and finite-dimensional. In fact, we will use the notations in Subsection \ref{sec:5a}. Moreover we assume that $\hat U:=1\otimes U$, where $ U $ is a unitary corepresentation of the compact quantum group $ Q $ such that $ ( Q, U ) $ coacts by orientation and real structure preserving isometries on the finite dimensional spectral triple $(\A_2,\HH_2,D_2,\gamma_2,J_2) .$ Under these assumptions, we now establish the invariance for the bosonic part.

\begin{lemma}\label{lemma:traceB}
For any trace-class operator $L$ on $\HH=\HH_1\otimes\HH_2$
$$
( \tr_{\HH} \otimes {\rm id} ) \hat{U}(L\otimes 1)\hat{U}^*=\tr_{\HH}(L)\cdot 1_Q\ .
$$
\end{lemma}

\begin{proof}
Let $L=L_1\otimes L_2$ with $L_1\in\mathcal{L}^1(\HH_1)$ and $L_2\in\B(\HH_2).$ Since
$$
\hat U(L\otimes 1)\hat U^*=L_1\otimes U(L_2\otimes 1)U^*,
$$
by Lemma \ref{lemma:trace}, we have:
\begin{align*}
( \tr_{\HH_1\otimes\HH_2} \otimes {\rm id} ) \hat{U}(L\otimes 1)\hat{U}^* &=
\tr_{\HH_1}(L_1)\cdot ( \tr_{\HH_2} \otimes {\rm id} )\,U(L_2\otimes 1)U^*\cdot 1_Q \\
&=\tr_{\HH_1\otimes\HH_2}(L)\cdot 1_Q \;.
\end{align*}
Since $\HH_2$ is finite dimensional, any element of $\mathcal{L}^1(\HH_1\otimes\HH_2)$ is
a finite sum of elements of the form $L:=L_1\otimes L_2$, with $L_1\in\mathcal{L}^1(\HH_1)$ and $L_2 \in\B(\HH_2), $ and thus by the linearity of the trace, the proof is finished.
\end{proof}

\begin{prop} \label{invariance-bosonic}
For any $A\in \Omega^{1,s.a.}_D$,
$\tilde S_b[\mathrm{Ad_{\hat{U}}}(A)]=S_b[A]\cdot 1_Q$.
\end{prop}

\begin{proof}
From \eqref{eq:UD} we have
\begin{align*}
\tilde S_b[\hat U(A\otimes 1)\hat U^*] &= ( \tr_{\HH} \otimes {\rm id} ) \,f(D_{\hat U(A\otimes 1)\hat U^*}/\Lambda) \\
&= ( \tr_{\HH} \otimes {\rm id} ) \,f\bigl(\hat U(D_A\otimes 1)\hat U^*/\Lambda\bigr) \;.
\end{align*}
By continuous functional calculus,
$$
f\bigl(\hat U(D_A\otimes 1)\hat U^*/\Lambda\bigr)=\hat Uf\bigl((D_A\otimes 1)/\Lambda\bigr)\hat U^*
=\hat U\bigl(f(D_A/\Lambda)\otimes 1\bigr)\hat U^*
$$
and applying Lemma \ref{lemma:traceB} to the trace-class operator $L:=f(D_A/\Lambda)$ we get
\begin{align*}
\tilde S_b[\hat U(A\otimes 1)\hat U^*] &= ( \tr_{\HH} \otimes {\rm id} ) \,\hat{U}(L\otimes 1)\hat{U}^* \\
&=\tr_{\HH}(L)\cdot 1_Q \equiv \tr_{\HH}f(D_A/\Lambda)\cdot 1_Q \\
&=S_b[A]\cdot 1_Q \;,
\end{align*}
which concludes the proof.
\end{proof}

\begin{prop}
The bosonic and the fermionic part of the spectral action of the Standard Model are preserved by the compact
quantum group
$\widetilde{\rule{0pt}{7pt}\smash[t]{\mathrm{QISO}}}\rule{0pt}{8pt}^+(B_F,H_F,D_F,\gamma_F,J_F)$.
\end{prop}

\begin{proof}
The compact quantum group
$Q:=\widetilde{\rule{0pt}{7pt}\smash[t]{\mathrm{QISO}}}\rule{0pt}{8pt}^+
(B_F,H_F,D_F,\gamma_F,J_F)$ 
has a corepresentation preserving $\HH_+$ and it satisfies the hypothesis of Lemma \ref{lemma:product}
and Prop.~\ref{invariance-fermionic} and \ref{invariance-bosonic}, hence the result
follows.
\end{proof}


\section{Some remarks on real $*$-algebras and their symmetries}\label{sec:6}

In Sec.~\ref{sec:iso} we computed the quantum isometry group of the finite
part of the Standard Model by replacing the real $ C^* $-algebra $A_F$ with
the complex $ C^* $-algebra $B_F$. Here we explain what happens if we
work with $ A_F. $

Any real $*$-algebra $\A$ (i.e.~unital, associative, involutive algebra over $\R$)
can be thought of as the fixed point
subalgebra of its complexification $\A_{\C}=\A\otimes_{\R}\C$
with respect to the involutive (conjugate-linear) real $*$-algebra
automorphism $ \sigma $ defined by
\begin{equation}\label{eq:cansigma}
\sigma(a\otimes_{\R}z)=a\otimes_{\R}\bar{z}\quad\forall\;a\in\A,z\in\C \;,
\end{equation}
that is
$$
\A=\{a\in\A_{\C}:\sigma(a)=a\} \;.
$$
A crucial observation is that we can characterize the automorphisms
of $\A$ as those automorphisms of $\A_{\C}$ which commute with $\sigma$,
as proved in the following lemma.

\begin{lemma}\label{lemma:autSigma}
For any real $*$-algebra $\A$,
\begin{equation}\label{eq:autsigma}
\mathrm{Aut}(\A)\simeq\bigl\{\phi\in\mathrm{Aut}(\A_{\C}) \;:\;\sigma\phi=\phi\hspace{1pt}\sigma\bigr\} \;.
\end{equation}
\end{lemma}
\begin{proof}
If $\varphi$ is any (real) $*$-algebra morphism of $\A$, $\phi(a\otimes_{\R}z):=
\varphi(a)\otimes_{\R}z$ defines a (complex) $*$-algebra morphism of $\A_{\C}$
clearly satisfying $\sigma\phi=\phi\hspace{1pt}\sigma$. The map $\varphi\mapsto\phi$
gives an inclusion of the left hand side of \eqref{eq:autsigma} into the right hand side.
Conversely, if $\phi\in\mathrm{Aut}(\A_{\C})$ satisfies
$\sigma\phi=\phi\hspace{1pt}\sigma$, then it maps the real subalgebra
$\A\simeq\A\otimes_{\R}1\subset\A_{\C}$ into itself, since
$$
\sigma\phi(a\otimes_{\R}1)=\phi\hspace{1pt}\sigma(a\otimes_{\R}1)=\phi(a\otimes_{\R}1)
$$
for any $a\in\A$. Therefore, we can define an element $\varphi\in\mathrm{Aut}(\A)$
by $\varphi(a)\otimes_{\R}1:=\phi(a\otimes_{\R}1)$.

The two group homomorphisms $\varphi\mapsto\phi$ and $\phi\mapsto\varphi$
are the inverses of each other and thus, we have the isomorphism in \eqref{eq:autsigma}.
\end{proof}

\noindent
From a dual point of view,
if $G=\mathrm{Aut}(\A)$, the right coaction of $C(G)$ on $\A_{\C}$
is the map $\alpha:\A_{\C}\to\A_{\C}\otimes C(G)\simeq C(G;\A_{\C})$ defined by
$$ ( {\rm id} \otimes {\rm ev}_{\phi} ) \alpha(a):=\phi(a), ~ \phi ~ \in  G, ~ a ~ \in  \A_{\C} \;.
$$
We can rephrase Lemma \ref{lemma:autSigma} as follows.

\begin{lemma}
For a finite dimensional real $ C^* $-algebra $\A$, the condition $\sigma\,\phi=\phi\,\sigma\;\forall\;\phi\in G$ is equivalent to
$$
(\sigma\otimes \ast_{C ( G ) })\alpha=\alpha\sigma \;.
$$
\end{lemma}
\begin{proof}
Let $ \alpha_\phi \sigma =  ( \sigma \otimes ev_{\phi} ~ \ast_{C ( G ) } ) \alpha $ and $ \phi \in G, ~ a \in \A_{\C}. $ Let us suppose that $(\sigma\otimes \ast_{C ( G ) })\alpha=\alpha\sigma $.  Then $ \sigma \phi ( a ) = ( {\rm id} \otimes  ev_{\phi} ) \alpha \sigma ( a )  = ( \sigma \otimes  ev_{\phi}~ \ast_{C ( G ) } ) \alpha ( a ) = ( \sigma \otimes \ast_{\C} ~ {\rm ev}_{\phi} ) \alpha ( a ) = \phi \sigma ( a ) $  by the antilinearity of $ \sigma. $
Conversely, if $ \sigma\,\phi=\phi\,\sigma\;\forall\;\phi\in G $ then for all $ \phi, ~ ( {\rm id} \otimes {\rm ev}_{\phi} )  \alpha ( \sigma ( a ) ) = ( \sigma \otimes {\rm ev}_{\phi} ) ( \alpha ( a ) ).$
Thus, $  ( \sigma \otimes  {\rm ev}_{\phi} ~ \ast_{C ( G ) } ) \alpha ( a ) = ( \sigma \otimes \ast_{\C} ~ {\rm ev}_{\phi} ) \alpha ( a ) = \sigma ( ( {\rm id} \otimes {\rm ev}_{\phi} ) \alpha ( a ) ) = \sigma \phi ( a ) = \phi \sigma ( a ) = ( {\rm id} \otimes {\rm ev}_{\phi} )  \alpha ( \sigma ( a ) ).$
As $ \{ {\rm ev}_{\phi} : \phi \in G \}  $ separates points on $ G, $ this proves $ (\sigma\otimes \ast_{C ( G ) })\alpha=\alpha\sigma $.
\end{proof}

Motivated by this lemma, we consider the category $\catB$ of CQGs coacting by orientation and real structure
preserving isometries via a unitary corepresentation $U$ (in the sense of Def.~\ref{def:qisot})
on the spectral triple $(B_F,H_F,D_F,\gamma_F,J_F)$ whose adjoint coaction $\mathrm{Ad_U}$ can be extended to a
coaction $\alpha$ on $(A_F)_{\C}=A_F\otimes_{\R}\C$ satisfying
\begin{equation}\label{eq:alphasigma}
(\sigma\otimes *)\alpha=\alpha\sigma \;.
\end{equation}
We notice that it is a subcategory
of $\catA$: objects of $\catB$ are those objects of $\catA$ compatible with $\sigma$ in the
sense explained above, and the morphisms in the two categories are the same.

Thus any object, say $ Q, $ of $\catB$ satisfies the relations of the universal object $\qisot( D_F )$ of $\catA$  in Prop.~\ref{prop:main}. In the rest of this subsection, with a slight abuse of notation, we will continue to denote the generators of $ Q $ by the same symbols as in Prop.~\ref{prop:main}.

\begin{thm}\label{thm:real}
A compact quantum group $Q$ is an object in $\catB$ if and only if the generators satisfy
\begin{equation}\label{eq:realrel}
(T_m)_{jk}(T_m)_{j'k'}^*(T_m)_{j''k''}=
(T_m)_{j''k''}(T_m)_{j'k'}^*(T_m)_{jk}
\end{equation}
for all $m=1,\ldots,n$ and all $j,j',j'',k,k',k''\in\{1,2,3\}$.
\end{thm}
\begin{proof}
The real algebra $A_F=\C\oplus\Q\oplus M_3(\C)$ is the fixed point subalgebra of
$(A_F)_{\C}\simeq\C\oplus\C\oplus M_2(\C)\oplus M_3(\C)\oplus M_3(\C)$
with respect to the automorphism
$$
\sigma(\lambda,\lambda',q,m,m')=
(\bar\lambda\hspace{1pt}',\bar\lambda,\sigma_2\bar q\sigma_2,\bar m\hspace{1pt}',\bar m) \;,
$$
where $\sigma_2$ is the second Pauli matrix:
$$
\sigma_2:=
\begin{mat}
0 & -i \\ i & 0
\end{mat} \;
$$
It is easy to check that $q\in M_2(\C)$ satisfies $\sigma_2\bar q\sigma_2=q$
if an only if it is of the form \eqref{eq:XXX}, and that under the isomorphism
\eqref{eq:isoinv} $\C$ is identified with the real subalgebra of $\C\oplus\C$
with elements $(\lambda,\bar\lambda)$ and $M_3(\C)$ with the
real subalgebra of $M_3(\C)\oplus M_3(\C)$ with elements $(m,\bar m)$.

The coaction on the factor $B_F\subset (A_F)_{\C}$ is given by \eqref{eq:Ucoact},
and an extension $\widetilde{\mathrm{Ad_U}}$ to $(A_F)_{\C}$ satisfying \eqref{eq:alphasigma} exists if and only if
\begin{align*}
\widetilde{\mathrm{Ad_U}}(0,0,0,0,e_{ij}) &=(\sigma\otimes *)\widetilde{\mathrm{Ad_U}}\,\sigma(0,0,0,0,e_{ij}) \\
&=(\sigma\otimes *)\widetilde{\mathrm{Ad_U}} (0,0,0,\bar e_{ij},0) \\
&=(\sigma\otimes *)\bigl(\mathrm{Ad_U}(\inner{0,0,0,e_{ij}}),0\bigr) \\
&=(\sigma\otimes *)\sum\nolimits_{k,l=1,2,3}(0,0,0,e_{kl},0)\otimes  ( T_1 )_{ k i}^* (T_1)_{ l j} \\
&=\sum\nolimits_{k,l=1,2,3}(0,0,0,0,e_{kl})\otimes  ( T_1 )_{ l j}^* (T_1)_{ k i} \;.
\end{align*}
The only conditions left to impose is that this extension is a coaction of a CQG. As it is already a coaction on $ B_F, $ we need to impose  it for the coaction on the second copy of $ M_3 ( \C ), $ which has to be preserved by $ \widetilde{\mathrm{Ad_U}}. $ At this point, we note that as $ \widetilde{\mathrm{Ad_U}} $ is an extension of $ \mathrm{Ad_U}, $ which preserves the trace on the first copy of $ M_3 ( \C ), $ the formula $ \widetilde{\mathrm{Ad_U}}(0,0,0,0,e_{ij}) = \sum\nolimits_{k,l=1,2,3}(0,0,0,0,e_{kl})\otimes  ( T_1 )_{ l j}^* (T_1)_{ k i} $ forces $ \widetilde{\mathrm{Ad_U}} $ to preserve the trace on the second copy of  $ M_3 ( \C ). $ Thus, by Theorem 4.1 of \cite{Wan98a}, it suffices to impose the conditions (4.1-4.5) in that paper with $ a^{kl}_{ij} $ replaced by  $ ( T_m )_{ l j}^* (T_m)_{ k i}. $ It is easy to check that (4.3-4.5) are automatically satisfied. The only non trivial conditions come from (4.1) and (4.2).

From (4.1), we get
\begin{equation}\label{eq:realrelB}
\sum\nolimits_{v=1}^3(T_m)^*_{vj}(T_m)_{ki}(T_m)_{ls}^*(T_m)_{vr}=
\delta_{jr}(T_m)_{ls}^*(T_m)_{ki}
\end{equation}
From (4.2), we get the same relation with  $(T_m)^t$
instead of $T_m.$  Now we show that \eqref{eq:realrelB} and \eqref{eq:realrel}
are equivalent, which will finish the proof since if $T_m$ satisfies \eqref{eq:realrel},
then $(T_m)^t$ satisfies it too.

If we multiply both sides of \eqref{eq:realrelB} by $(T_m)_{qj}$ from the left
and sum over $j$, we get
$$
\sum\nolimits_{v=1}^3\delta_{vq}(T_m)_{ki}(T_m)_{ls}^*(T_m)_{vr}=
\sum\nolimits_{j=1}^3\delta_{jr}(T_m)_{qj}(T_m)_{ls}^*(T_m)_{ki}
$$
using biunitarity of $T_m$. The last equation is clearly equivalent to \eqref{eq:realrel}.
To prove that \eqref{eq:realrel} implies \eqref{eq:realrelB}, it is enough to multiply
both sides by $(T_m)_{j''k'''}$
from the left, then sum over $j''$ and use the biunitarity of $T_m$ again.
\end{proof}

It is easy to check that \eqref{eq:realrel} defines a Woronowicz $C^*$-ideal, and hence
the quotient of $\qisot( D_F )$ by \eqref{eq:realrel} is a CQG. This leads to the following corollary.

\begin{cor}\label{thm:univ}
Let $\,\qisotr ( D_F )$ be the quantum subgroup of the CQG $\,\qisot ( D_F )$ in Prop.~\ref{prop:main}
defined by the relations \eqref{eq:realrel}.
Then $\qisotr ( D_F )$ is the universal object in the category $\catB$.
\end{cor}

Motivated by \eqref{eq:realrel}, we give the following definition.

\begin{df}
For a fixed $N$, we call $A_u^*(N)$ the universal unital $C^*$-algebra generated
by a $N\times N$ biunitary $u=((u_{ij}))$ with relations
\begin{equation}\label{eq:fine}
ab^*c=cb^*a \;,\qquad\forall\;a,b,c\in\{ u_{ij},\,i,j=1,\ldots,N\} \;.
\end{equation}
$A_u^*(N)$ is a CQG with coproduct given by $\Delta(u_{ij})=\sum\nolimits_k u_{ik} \otimes u_{kj}$.
\end{df}

We will call $A_u^*(N)$ the $N$-dimensional \emph{half-liberated unitary group}.
This is similar to the half-liberated orthogonal group $A_o^*(N)$,
that can be obtained by imposing the further relation $a=a^*$ for all $a
\in\{ u_{ij},\,i,j,=1,\ldots,N\}$ (cf.~\cite{BV09}).

\begin{rem}
We notice that there are two other possible ways to ``half-liberate'' the free unitary group.
Instead of $ab^*c=cb^*a$ (which by adjunction is equivalent to $a^*bc^*=c^*ba^*$),
one can consider respectively the relation $a^*bc=cba^*$ (which is equivalent to
$abc^*=c^*ba$ and to the adjoints $ab^*c^*=c^*b^*a$ and $a^*b^*c=cb^*a^*$)
or $abc=cba$ (equivalent to $a^*b^*c^*=c^*b^*a^*$) for any triple
$a,b,c\in\{ u_{ij},\,i,j=1,\ldots,N\}$.
\end{rem}

Like $A_o^*(N),$ the projective
version of $A_u^*(N)$ is also commutative, as proved in the next proposition.

\begin{prop}\label{prop:puk}
The CQG $PA_u^*(N) $ is isomorphic to  $ C(PU(N))$.
\end{prop}

\begin{proof}
We recall (Rem.~\ref{remarkprojectiveversion}) that for a CQG  $Q$  generated by a biunitary $u=((u_{ij}))$, the projective version
 is the $C^*$-subalgebra generated by products $u_{ij}^*u_{kl}$.

Clearly $ C ( U(N) )$ is a quantum subgroup of $A_u^*(N)$, and the latter is a quantum subgroup of $A_u(N)$.
Thus,
$C ( PU(N) )$ is a quantum subgroup of $PA_u^*(N)$, which is a quantum subgroup of $PA_u(N)$.
Since the abelianization of $PA_u(N)$ is exactly $C ( PU(N) )$, any commutative (as a $ C^* $-algebra) quantum subgroup of
$PA_u(N)$ containing $C ( PU(N) )$ coincides with $C ( PU(N) )$. Thus, the proof will be over if we can show that
the $C^*$-algebra of $PA_u(N)$ is commutative, i.e.~$PA_u(N)$ is the space of continuous functions on a compact group.
This is a simple computation.
Using first \eqref{eq:fine} and then its adjoint we get:
\begin{align*}
(u_{ij}^*u_{kl})(u_{pq}^*u_{rs}) &=
u_{ij}^*(u_{kl}u_{pq}^*u_{rs}) =
u_{ij}^*(u_{rs}u_{pq}^*u_{kl}) \\ &=
(u_{ij}^*u_{rs}u_{pq}^*)u_{kl} =
(u_{pq}^*u_{rs}u_{ij}^*)u_{kl} \\ &=
(u_{pq}^*u_{rs})(u_{ij}^*u_{kl}) \;.
\end{align*}
This proves that the generators of $PA_u(N)$ commute, which concludes the proof.
\end{proof}

\noindent
In complete analogy with \eqref{eq:qsubG}, if we call
$Q_n^*(n')$ the amalgamated free product of $ n $ copies of $ A_u^* (n') $ over the common \mbox{Woronowicz} \mbox{$ C^* $-subalgebra} $C(PU(n'))$, then we have:

\begin{cor}\label{thm:mainreal}
$\qisotr ( D_F )$ is a quantum subgroup of the free product
$$
\underbrace{C(U(1))*C(U(1))*\ldots*C(U(1))}_{n+1}
\,*\,\,Q_n^*(3) *A_u(n)
$$
The Woronowicz $C^*$-ideal of this CQG defining $\qisotr ( D_F )$ is determined by \eqref{eq:qisotA} and \eqref{eq:qisotC}.
\end{cor}

\noindent
As in the complex case, let us denote by $\qisor( D_F )$ the $ C^* $-subalgebra of $\qisotr ( D_F )$ generated by 
$\inner{\xi\otimes 1, \mathrm{Ad_{U_{\R}}} ( a )(\eta\otimes 1)}$, where $a\in B_F$,
$\xi,\eta\in H_F$ and
$ U_{\R} $ is the corepresentation of $\qisotr ( D_F ).$ An immediate corollary of Prop.~\ref{prop:puk}
and Corollary \ref{thm:mainreal} is the following.

\begin{cor}
$\qisor ( D_F ) =C(U(1))*C(PU(3))$.
\end{cor}

\begin{rem}
Since $\qisotr ( D_F )$ is a quantum subgroup of $\qisot ( D_F )$, its coaction still preserves the
spectral action.
\end{rem}

A detailed study of quantum automorphisms for finite-dimensional real $C^*$-algebras,
along the lines of the discussion in this section, will be reported elsewhere.


\section{Proof of Proposition {\protect\ref{prop:main}}}\label{sec:7}
In this section, we prove the main result, that is, Proposition \ref{prop:main}. Throughout this section, $ ( Q, U )  $ will denote an object in $\catA.$ We start by exploiting the conditions regarding $ \gamma_F $ and $ J_F $, then we  use the conditions regarding $ D_F $ and $ \mathrm{Ad_U} $   to get a neater expression for $ U $ in Lemma 6.2, 6.3 and 6.4 and then using these simplified expressions in the next Lemmas, we derive the desired form of $ U $ from which we can identify the quantum isometry group. We will use Remark \ref{Upsilon_positive} in this section without mentioning it.
Recall that $\B(H_F)=M_2(\C)\otimes M_4(\C)\otimes M_4(\C)\otimes M_n(\C)$, where $n$ is the number of generations.

\begin{lemma}\label{lemma:appendix1}
$U\in \B(H_F)\otimes Q$ satisfies
$(\gamma_F\otimes 1)U=U(\gamma_F\otimes 1)$
and $(J_0\otimes 1)\bar U=U(J_0\otimes 1)$ if{}f
\begin{align}
U &=\sum\nolimits_{IJ}
(e_{i_1j_1}\otimes e_{i_2j_2}\otimes e_{i_3j_3}\otimes e_{i_4j_4})\otimes u_{IJ}
\notag\\[2pt] & \qquad
+\sum\nolimits_{IJ}
(e_{i_1j_1}\otimes e_{i_2j_2}\otimes e_{i_3+2,j_3+2}\otimes e_{i_4j_4})\otimes \bar{u}_{IJ}
\;,\label{eq:Uprime}
\end{align}
where the multi-indices $I=(i_1,\ldots,i_4)$, $J=(j_1,\ldots,j_4)$, etc.~run
in $\{1,2\}\times\{1,2,3,4\}\times\{1,2\}\times\{1,2,\ldots,n\}$.
\end{lemma}

\begin{proof}
The condition $(\gamma_F\otimes 1)U=U(\gamma_F\otimes 1)$ implies that $ u_{i_1,j_1, i_2, j_2, i_3, j_3, i_4, j_4} = 0 $ unless
$i_3, j_3$ are both 
greater or equal than $2$ or both less or equal than $3$.
Using the reduced form of $ U $ obtained from
this observation, we impose $(J_0\otimes 1)\bar U=U(J_0\otimes 1)$ and get 
$u_{i_1,j_1, i_2, j_2, i_3, j_3, i_4, j_4}=(u_{i_1,j_1, i_2, j_2, i_3-2, j_3-2, i_4, j_4})^*$
for all $i_3,j_3\geq 3$, which proves the Lemma.
\end{proof}

Let $ V_1, V_2, V_3, V_4 $ denote the subspaces $(e_{11}\otimes e_{11}\otimes 1\otimes 1)\HH$,
$(e_{22}\otimes e_{11}\otimes 1\otimes 1)\HH$,
$(e_{11}\otimes (1-e_{11})\otimes 1\otimes 1)\HH$,
and
$(e_{22}\otimes (1-e_{11})\otimes 1\otimes 1)\HH$ respectively.

\begin{lemma}\label{lemma:appendix2}
If $ U $ is of the form \eqref{eq:Uprime} and
 commutes with $ D_F, $ the subspaces $ V_i, ~ i = 1,2,3,4 $ are kept invariant by U and thus  \eqref{eq:Uprime} becomes
\begin{align}
U &=
\sum_{i=1,2}
e_{ii}\otimes e_{11}\otimes
\begin{mfour}
\alpha_{11}^i & \alpha_{12}^i & 0 & 0 \\
\alpha_{21}^i & \alpha_{22}^i & 0 & 0 \\
0 & 0 & \bar\alpha_{11}^i & \bar\alpha_{12}^i \\
0 & 0 & \bar\alpha_{21}^i & \bar\alpha_{22}^i
\end{mfour}
\notag\\[2pt]
& \qquad +\sum_{\substack{i=1,2\\ j,k=1,2,3}}e_{ii}\otimes
e_{j+1,k+1}\otimes
\begin{mfour}
\beta_{11}^{\,i,j,k} & \beta_{12}^{\,i,j,k} & 0 & 0 \\
\beta_{21}^{\,i,j,k} & \beta_{22}^{\,i,j,k} & 0 & 0 \\
0 & 0 & \bar\beta_{11}^{\,i,j,k} & \bar\beta_{12}^{\,i,j,k} \\
0 & 0 & \bar\beta_{21}^{\,i,j,k} & \bar\beta_{22}^{\,i,j,k}
\end{mfour}
\label{eq:Usecond}
\end{align}
where, as in \eqref{eq:Dirac} we identify
$M_4(\C)\otimes M_n(\C)\otimes Q$ with $M_{4n}(Q)$,
we call
$\alpha^i_{j_1k_1}$ is the $n\times n$ matrix
with entries $(\alpha^i_{j_1k_1})_{j_2k_2}:=
u_{JK}$ with $J=(i,1,j_1,j_2)$ and $K=(i,1,k_1,k_2)$
and we call
$\beta_{j_1k_1}^{\,i,j_0,k_0}$ the $n\times n$ matrix
with entries $(\beta_{j_1k_1}^{\,i,j_0,k_0})_{j_2k_2}:=
u_{JK}$ with
$J=(i,j_0+1,j_1,j_2)$ and $K=(i,k_0+1,k_1,k_2)$.
\end{lemma}

\begin{proof}
 The subspaces $ V_i, i = 1,2,3,4 $ are $D_F$-invariant and correspond to distinct
sets of eigenvalues (masses of the generations of $\nu$,
$e$, $u$ and $d$ respectively). Since $(D_F\otimes 1)U=U(D_F\otimes 1)$
these four subspaces must be preserved by $U$ and this completes the proof of the lemma.
\end{proof}

\begin{lemma}\label{lemma:split}
Let $Q$ be any CQG with $U$ as in \eqref{eq:Uprime} and satisfying \eqref{eq:cond6}.
Then each one of the four summands in
$B_F=\C\oplus\C\oplus M_2(\C)\oplus M_3(\C)$ is a coinvariant
subalgebra under the adjoint coaction $\mathrm{Ad_U}(a)=U ( a \otimes 1 ) U^*$ of $Q$.
\end{lemma}

\begin{proof}
We start with the basis element $\inner{0,1,0,0}$ of the second copy of $\C$.
Equation \eqref{eq:cond6} means that
\begin{align} \label{8888}
\mathrm{Ad_U}(\inner{0,1,0,0}) &=
\inner{1,0,0,0}\otimes a^{\inner{1,0,0,0}} +
\inner{0,1,0,0}\otimes a^{\inner{0,1,0,0}} \notag\\[2pt] &\qquad +
\sum_{i,j=1,2}\inner{0,0,e_{ij},0}\otimes a^{\inner{0,0,\smash[b]{e_{ij}},0}} +
\sum_{i,j=1,2,3}\inner{0,0,0,e_{ij}}\otimes a^{\inner{0,0,0,\smash[b]{e_{ij}}}} \;,
\end{align}
where $a^{\inner{.}}$ are some elements of $Q$.

By \eqref{eq:Usecond}, $U ( \inner{0,1,0,0} \otimes 1 )U^*$ has $e_{22}$ in the first position
and $e_{jk}$ in the third, with $j,k=3,4$.
Therefore, $U ( \inner{0,1,0,0} \otimes 1 ) U^*$ vanishes on the subspaces
$(e_{11}\otimes 1\otimes e_{44}\otimes 1)H_F$,
$(1\otimes 1\otimes e_{11}\otimes 1)H_F$
and
$(1\otimes (1-e_{11})\otimes (e_{22}+e_{33})\otimes 1)H_F$. Applying \eqref{8888} on these three subspaces and using \eqref{eq:B}
we get respectively:
\begin{align*}
0 &=(e_{11}\otimes 1\otimes e_{44}\otimes 1)\otimes a^{\inner{1,0,0,0}} +0+0+0
\;,\\
0&=0+0+\sum\nolimits_{i,j=1,2}\inner{0,0,e_{ij},0}\otimes a^{\inner{0,0,\smash[b]{e_{ij}},0}} +0
\;,\\
0&=0+0+0+\sum\nolimits_{i,j=1,2,3}\inner{0,0,0,e_{ij}}\otimes a^{\inner{0,0,0,\smash[b]{e_{ij}}}}
\;.
\end{align*}
Therefore $a^{\inner{1,0,0,0}}=
a^{\inner{0,0,\smash[b]{e_{ij}},0}}=a^{\inner{0,0,0,\smash[b]{e_{ij}}}}=0$
and
$\mathrm{Ad_U}(\inner{0,1,0,0})\subset\inner{0,1,0,0}\otimes Q$.
The proof for the other three factors is similar.

For the rest of the proof, let $ \lambda \in \C, ~ q \in M_2(  \C ), m \in M_3 ( \C ) $ be arbitrary.

$U ( \inner{1,0,0,0} \otimes 1 ) U^*$ vanishes on the subspaces
$(e_{22}\otimes (1-e_{11})\otimes e_{44}\otimes 1) H_F$,
$(1\otimes e_{11}\otimes e_{11}\otimes 1) H_F$ and
$(1\otimes (1-e_{11})\otimes e_{22} \otimes 1) H_F$ and hence this implies
respectively that the coefficients of $ \inner{0,\lambda,0,0},~ \inner{0,0,q,0},~ \inner{0,0,0,m} $ in $ \mathrm{Ad_U}(\inner{1,0,0,0})$ are zero.

$U (\inner{0,0,q,0} \otimes 1 ) U^*$ vanishes on the subspaces
$(e_{11}\otimes 1\otimes e_{44}\otimes 1) H_F$,
$(e_{22}\otimes 1\otimes e_{44}\otimes 1) H_F$ and
$(1\otimes (1-e_{11})\otimes e_{33}\otimes 1) H_F$ and hence this implies
respectively that the coefficients of $ \inner{\lambda,0,0,0},~ \inner{0,\lambda,0,0},~ \inner{0,0,0,m} $ in $\mathrm{Ad_U}(\inner{0,0,q,0})$ are zero.

Finally, $U(\inner{0,0,0,m} \otimes 1) U^*$ vanishes on the subspaces
$(e_{11}\otimes e_{11}\otimes e_{44}\otimes 1) H_F$,
$(e_{22}\otimes e_{11}\otimes e_{44}\otimes 1) H_F$ and
$(1\otimes e_{11}\otimes e_{11}\otimes 1) H_F$ which implies
respectively that the coefficients of $ \inner{\lambda,0,0,0},~ \inner{0,\lambda,0,0},~ \inner{0,0,q,0} $ in $\mathrm{Ad_U}(\inner{0,0,0,m})$ are zero.
\end{proof}

\begin{lemma}\label{lemma:vanish}
If \eqref{eq:cond6} is satisfied, the matrices $\alpha^i_{j_1k_1}$
and $\beta_{j_1k_1}^{\,i,j_0,k_0}$ in \eqref{eq:Usecond}
are zero for all $j_1\neq k_1$.
\end{lemma}

\begin{proof}
We use Lemma \ref{lemma:split}.
Since $\mathrm{Ad_U}(\inner{1,0,0,0})\subset \inner{1,0,0,0}\otimes Q$, it is easy to see that  $ ( e_{ii}\otimes e_{11}\otimes e_{11}\otimes 1\otimes 1_Q ) \mathrm{Ad_U}(\inner{1,0,0,0})$ equals zero for all $i=1,2$.
On the other hand,  straightforward computation gives
$$
(e_{ii}\otimes e_{11}\otimes e_{11}\otimes 1\otimes 1_Q)
\mathrm{Ad_U}(\inner{1,0,0,0})=
e_{ii}\otimes e_{11}\otimes e_{11}\otimes \alpha^i_{12}(\alpha^i_{12})^*=0 \;,
$$
from which it follows that $\alpha^i_{12}=0$ for all $i=1,2$.
Similarly,
$$
(1\otimes e_{11}\otimes e_{22}\otimes 1\otimes 1_Q)
\mathrm{Ad_U}(\inner{0,0,e_{ii},0})=
e_{ii}\otimes e_{11}\otimes e_{22}\otimes \alpha^i_{21}(\alpha^i_{21})^*=0
$$
gives $\alpha^i_{21}=0$ for all $i=1,2$. Finally,
 $\mathrm{Ad_U}(\inner{0,0,0,e_{k_0l_0}})$ applied to the projections
$1\otimes 1\otimes e_{11}\otimes 1$ and
$1\otimes 1\otimes e_{44}\otimes 1,$ we get the conditions
$$
\beta_{12}^{i,j_0,k_0}(\beta_{12}^{i,l_0,n_0})^*=
\bar\beta_{21}^{i,j_0,k_0}(\beta_{21}^{i,l_0,n_0})^t=0
$$
for all $i,j_0,k_0,l_0,n_0$.
In particular setting $j_0=l_0$ and $k_0=n_0$ we get
$\beta_{12}^{i,j_0,k_0}=\beta_{21}^{i,j_0,k_0}=0$.
\end{proof}

\smallskip

Now we impose $U(D_F\otimes 1)=(D_F\otimes 1)U$, with
$D_F$ as in \eqref{eq:Dirac}, $ U $ as in \eqref{eq:Usecond}
and using Lemma \ref{lemma:vanish}.

\begin{lemma}\label{lemma:5}
Any $U\!$ of the form \eqref{eq:Usecond}, and with $\alpha^i_{j_1k_1}=\beta_{j_1k_1}^{\,i,j_0,k_0}=0$
for all $j_1\neq k_1$,
satisfies
$U(D_F\otimes 1)=(D_F\otimes 1)U\!$ if and only if
\begin{enumerate}
\item all $\alpha^2_{ss}$ and $\beta^{\,1,j,k}_{rr}$
are diagonal $n\times n$
matrices,
\item $\alpha^2_{22}=\bar\alpha^2_{11}, ~ \beta^{\,1,j,k}_{22}=\bar\beta^{\,1,j,k}_{11}, ~ \beta^{2,j,r}_{22} = \bar{\beta^{2,j,r}_{11}}$,
\item $\alpha^1_{11}\Upsilon_\nu =\Upsilon_\nu \alpha^1_{11}=\bar\alpha^1_{22}\Upsilon_\nu=\Upsilon_\nu\,\bar\alpha^1_{22}$,
$\alpha^1_{22}\Upsilon_R =\Upsilon_R\,\bar\alpha^1_{22}$
\item $ {C}^*   \beta^{2,j,k}_{11}  C   $  is a diagonal matrix.
\end{enumerate}
\end{lemma}

\begin{proof}
The condition $U(D_F\otimes 1)=(D_F\otimes 1)U$ is equivalent to
the following sets of equations:
\begin{subequations}
\begin{align}
\alpha^1_{11}\Upsilon_\nu &=\Upsilon_\nu\,\bar\alpha^1_{22} \;,&
\bar\alpha^1_{22}\Upsilon_\nu &=\Upsilon_\nu\,\alpha^1_{11} \;,&
\alpha^1_{22}\Upsilon_R &=\Upsilon_R\,\bar\alpha^1_{22} \label{eq:relAlpha}
\;,\\[2pt]
\alpha^2_{11}\Upsilon_e &=\Upsilon_e\hspace{1pt}\bar\alpha^2_{22} \;,&
\alpha^2_{22}\Upsilon_e &=\Upsilon_e\hspace{1pt}\bar\alpha^2_{11} \;,&
\beta^{1,j,k}_{11}\Upsilon_u &=\Upsilon_u\hspace{1pt}\bar\beta^{1,j,k}_{22} \label{eq:condsolved}
\;,\\[2pt]
\bar\beta^{1,j,k}_{22}\Upsilon_u &=\Upsilon_u\,\beta^{1,j,k}_{11} \;,&
\beta^{2,j,k}_{11}\Upsilon_d &=\Upsilon_d\hspace{1pt}\bar\beta^{2,j,k}_{22} \;,&
\bar\beta^{2,j,k}_{22}\Upsilon_d &=\Upsilon_d\hspace{1pt}\beta^{2,j,k}_{11} \label{eq:condsolvedfivefive}
 \;,
\end{align}
\end{subequations}
Actually, there are additional $9$ relations that --- recalling that $\Upsilon_x$ ($x=e,u,d,\nu$)
are positive, $ \Upsilon_e, \Upsilon_u $ are diagonal and $\Upsilon_R$ is symmetric --- turn out to be the
``bar'' of previous ones and hence they do not give any new information.

From the first two equations in \eqref{eq:relAlpha}, we deduce
that $\alpha^1_{11}$ commute with $\Upsilon_\nu^2$:
$$
\alpha^1_{11}\Upsilon_\nu^2=\Upsilon_\nu\bar\alpha^1_{22}\Upsilon_\nu
=\Upsilon_\nu^2\alpha^1_{11} \;,
$$
and hence it commutes with its positive square root $\Upsilon_\nu$. Similarly
$\alpha^1_{11}$ commutes with $\Upsilon_\nu$ and the conditions
\eqref{eq:relAlpha} turn out to be equivalent to point 3. of the Lemma.

In a similar way 
from  \eqref{eq:condsolved} and \eqref{eq:condsolvedfivefive}
we deduce that all $\alpha^2_{ss}$ commute with
$\Upsilon_e^2$ and all $\beta^{\,1,j,k}_{rr}$ commute with $\Upsilon_u^2$.
Since $\Upsilon_x^2$ ($x=e,u$) are diagonal with distinct
eigenvalues, we deduce that all $\alpha^2_{ss}$ and
$\beta^{\,1,j,k}_{rr}$ must be diagonal $n\times n$
matrices. This proves 1.

As all $\alpha^2_{ss}$ and
$\beta^{\,1,j,k}_{rr}$ are diagonal, \eqref{eq:condsolved} implies that $\alpha^2_{22}=\bar\alpha^2_{11}$ and
$\beta^{\,1,j,k}_{22}=\bar\beta^{\,1,j,k}_{11},$ where we have used that $\Upsilon_e$ and $ \Upsilon_u $  are diagonal invertible matrices. Thus the first two equations of 2. are proved.

The second and third equation of \eqref{eq:condsolvedfivefive} implies respectively
\begin{equation}
  \beta^{2,j,k}_{22}  =  {(\Upsilon^t_d)}^{- 1}  \bar{\beta^{2,j,k}_{11}} \Upsilon^t_d,~  \beta^{2,j,k}_{22}  =  \Upsilon^t_d  \bar{\beta^{2,j,k}_{11}}  ( \Upsilon^t_d )^{ - 1 } .\label{jyotishmannew}
 \end{equation}

These two equations taken together means
   $$ {(\Upsilon^t_d)}^{- 1}  \bar{\beta^{2,j,k}_{11}} \Upsilon^t_d    =    \Upsilon^t_d   \bar{\beta^{2,j,k}_{11}}  ( \Upsilon^t_d )^{ - 1 }  .$$

Thus,
$ \beta^{2,j,k}_{11}   =  \Upsilon_d   {\Upsilon_d}^*  \beta^{2,j,k}_{11}  (  \Upsilon_d   {\Upsilon_d}^* )^{ - 1 }.  $
But,  $  \Upsilon_d   {\Upsilon_d}^* = C \delta_{\downarrow} {C}^* C {\delta_{\downarrow}}^* C^{\prime *} =  C \delta_{\downarrow}  {\delta_{\downarrow}}^* C^{\prime *}.  $
Thus, $ {C}^*   \beta^{2,j,k}_{11}  C   $ commutes with  $ \delta_{\downarrow}  {\delta_{\downarrow}}^*.$
As the latter is a diagonal matrix with distinct eigenvalues, $ {C}^*   \beta^{2,j,k}_{11}  C   $  is a diagonal matrix.
 The fact that  $ {C}^*   \beta^{2,j,k}_{11}  C   $  is diagonal implies that $ {(\Upsilon^t_d)}^{- 1}   \bar{\beta^{2,j,k}_{11}} \Upsilon^t_d = \Upsilon^t_d  \bar{\beta^{2,j,k}_{11}}  ( \Upsilon^t_d )^{ - 1 }. $ Thus, \eqref{jyotishmannew} is equivalent to 4.

The equation remaining to be proved is the third equation of 2. which follows by part 4. and \eqref{jyotishmannew}. Indeed, by \eqref{jyotishmannew},
  $\beta^{2,j,r}_{22} = \Upsilon^t_d \bar{\beta^{2,j,r}_{11}} ( \Upsilon^t_d )^{- 1}. = \bar{C} \delta_{\downarrow} ( C )^t \bar{\beta^{2,j,r}_{11}} \bar{C} ( \delta_{\downarrow} )^{- 1} ( C )^t. $

  As $ ( C )^* \beta^{2,j,r}_{11} C $ is diagonal by part 4., $ ( C )^t \bar{\beta^{2,j,r}_{11}} \bar{C}$ is also diagonal and hence it commutes with $ \delta_{\downarrow} $ and thus we get $\beta^{2,j,r}_{22} = \bar{\beta^{2,j,r}_{11}}.$

Conversely, if 1. - 4. of this Lemma are satisfied, then it can be easily verified that \eqref{eq:relAlpha}, \eqref{eq:condsolved} and \eqref{eq:condsolvedfivefive} are satisfied and hence $ U  $ commutes with $ D_F. $
\end{proof}

\noindent
In view of Lemma \ref{lemma:5}, we define elements $x_k$ and $3\times 3$ matrices $T_m$ by
$$
\alpha^2_{11} = \sum^n_{k = 1} e_{kk} \otimes x_k\;,
\qquad\quad
\beta^{\,1,j,k}_{11} = \sum^n_{m = 1} e_{mm} \otimes (T_m)_{j,k} \;.
$$
Hence, by part 2. of Lemma \ref{lemma:5},
$$
\beta^{\,1,j,k}_{22} = \sum^n_{m = 1} e_{mm} \otimes (\bar T_m )_{j,k} \;.
$$
Moreover, let
$$
X ( s, m ) =  \sum e_{ij} \otimes ( \beta^{2,i,j}_{11} )_{s,m} \;.
$$

\smallskip

\begin{lemma} \label{lemma:U11th}
If $U$ is a unitary corepresentation satisfying the hypothesis of Lemma \ref{lemma:5}, then
the matrices  $\alpha^i_{rr}$, $T_m$ and $X(m,m)$ are biunitaries. In particular,
$\{ x_1, x_2,....., x_n \} $ are unitary elements.
\end{lemma}

\begin{proof}
The condition $ U U^* = 1 \otimes 1 $ implies that for $ r = 1,2, $
$$ \alpha^i_{rr} ( \alpha^i_{rr} )^*  =  \bar\alpha^i_{rr} ( \bar\alpha^i_{rr} )^* = 1, $$
 $$ \sum_k \beta_{rr}^{\,i,j,k} ( \beta_{rr}^{\,i,l,k} )^* = \sum_k \bar\beta_{rr}^{\,i,j,k} ( \bar\beta_{rr}^{\,i,j,k} ) = \delta_{jl}.  $$
Similarly, from $ U^* U = 1 \otimes 1 $ we get the relations $$ ( \alpha^i_{rr} )^* \alpha^i_{rr}   =  ( \bar\alpha^i_{rr} )^* \bar\alpha^i_{rr}  = 1 ,$$
$$ \sum_k ( \beta_{rr}^{\,i,l,k} )^* \beta_{rr}^{\,i,j,k}  = \sum_k ( \bar\beta_{rr}^{\,i,l,k} )^* \bar\beta_{rr}^{\,i,j,k}  = \delta_{jl}.  .$$
Thus, the matrices  $ \alpha^i_{rr}$, $T_m$ and $X(m,m) $  are biunitaries.
\end{proof}

We note that in Lemma \ref{lemma:vanish} we provide a necessary condition for \eqref{eq:cond6}.
The next Lemma gives conditions that are necessary and sufficient.

\begin{lemma} \label{lemma:Utenth}
Assume $U$ satisfies the hypothesis of Lemma \ref{lemma:5} and \ref{lemma:U11th}.
The condition \eqref{eq:cond6} is satisfied, i.e.~the coaction
$\mathrm{Ad_U}$ preserves the subalgebra $B_F$,
if{}f there exists a unitary $x_0$  such that\vspace{-10pt}
\begin{subequations}
\begin{gather}
\alpha^1_{11} = {\rm diag} (x_0x_1,\ldots,x_0x_n ) \;, \quad\qquad
\alpha^2_{22} =  {\rm diag} (x_1^*,\ldots,x_n^*) \;, \label{eq:UtenthD}\\[6pt]
\beta^{2,i,j}_{11} = {\rm diag} ( x_0^* (T_1)_{i,j},\ldots, x_0^* (T_n)_{i,j} ) \;, \label{eq:UtenthA}\\[4pt]
\sum\nolimits_{m=1}^nC_{rm}\bar C_{sm}(T_m)_{j,k}=0 \qquad\forall\;r\neq s\quad (\,r,s=1,\ldots,n;\,j,k=1,2,3\,)\,, \label{eq:UtenthC} \\[4pt]
(T_m^*)_{i,j}(T_m)_{k,l}=(T_{m'}^*)_{i,j}(T_{m'})_{k,l} \qquad\forall\;m,m' \quad (\,i,j,k,l=1,2,3,\; m,m'=1,\ldots,n\,)\,, \label{eq:UtenthB} 
\end{gather}
\end{subequations}
and the adjoint coaction is
\begin{subequations}
\begin{align}
\mathrm{Ad_U}(\inner{0,0,e_{12},0}) &=\inner{0,0,e_{12},0}\otimes x_0, \label{eq:56a} \\
\mathrm{Ad_U}(\inner{0,0,e_{21},0}) &=\inner{0,0,e_{21},0}\otimes x_0^*, \label{eq:56b} \\
\mathrm{Ad_U}(\inner{0,0,0,e_{ij}}) &=\sum\nolimits_{kl}\inner{0,0,0,e_{kl}} \otimes ( (T_1)_{k,i})^*  (T_1)_{l,j} \;. \label{eq:56c}
\end{align}
\end{subequations}
Moreover, $\inner{1,0,0,0}$, $\inner{0,1,0,0}$
and $\inner{0,0,e_{ii},0}$ ($i=1,2$) are coinvariant.
\end{lemma}

\begin{proof}
We use the notations of the previous lemmas.
The coinvariance of $\inner{1,0,0,0}$, $\inner{0,1,0,0}$
and $\inner{0,0,e_{ii},0}$ ($i=1,2$) follows automatically
from unitarity of $U$.
Since
\begin{align*}
\mathrm{Ad_U}(\inner{0,0,e_{12},0}) &=
e_{12}\otimes e_{11}\otimes e_{11}\otimes\alpha_{11}^1(\alpha_{11}^2)^*
+\sum\nolimits_{ijk}e_{12}\otimes e_{i+1,k+1}\otimes e_{11}\otimes\beta_{11}^{1,i,j}
(\beta_{11}^{2,k,j})^*
\;,
\end{align*}
condition \eqref{eq:cond6} implies that there exists
$x_0\in Q$ such that
\begin{equation}\label{eq:beta11}
\alpha_{11}^1(\alpha_{11}^2)^*= \sum^n_{i = 1} e_{ii} \otimes x_0\;,
\qquad\quad
\sum_j \beta_{11}^{1,i,j}
(\beta_{11}^{2,k,j})^*=\delta_{i,k}  ( \sum^n_{i = 1} e_{ii} \otimes x_0 )\;.
\end{equation}
Unitarity of $\alpha^i_{rr}$ implies unitarity of $x_0$.  Moreover, we have $ \alpha^1_{11} = {\rm diag} ( x_0 x_1,\ldots, x_0 x_n )$.
\\
Using the relation $ \alpha^2_{11} = \bar{\alpha^2_{22}} $ in Lemma \ref{lemma:5}, we deduce that $\alpha^2_{22} = \sum^n_{k = 1} e_{kk} \otimes x^*_k$.
\\
We get $\mathrm{Ad_U}(\inner{0,0,e_{12},0})=\inner{0,0,e_{12},0}\otimes x_0$
and $\mathrm{Ad_U}(\inner{0,0,e_{21},0})=\inner{0,0,e_{21},0}\otimes x_0^*$.
\\
From the second equation of \eqref{eq:beta11},  we deduce that
$ \sum_j (T_m)_{i,j} ( \beta^{2,k,j}_{11} )^*_{sm} = \delta_{m,s} \delta_{i,k}x_0$.

Thus, $ \sum_j (  T_m )_{i,j} {(  X ( s, m )  )}^*_{k,j} = \delta_{m,s} \delta_{i,k} x_0  $.
and in particular $ T_m X ( s, m )^*  = \delta_{m,s} {\rm diag} ( x_0,x_0,x_0 )$,
which implies
$$
X ( s, m ) = 0 ~ {\rm if} ~  s \neq m ~ {\rm and} ~  X ( s, s ) = {\rm diag} ( x_0^*, x_0^*, x_0^* ) T_s,
$$
which translates into
\begin{equation} \label{beta2diagonaljyo}
               \beta^{2,i,j}_{11} = {\rm diag} ( x_0^* (T_1)_{i,j},\ldots, x_0^* (T_n)_{i,j}) \;.
\end{equation}
Moreover, as $ C^*\beta^{2}_{jk} C  $ is diagonal from 4. of Lemma \ref{lemma:5}, we get for all $ j,k = 1,2,3, $
 $$ \sum\nolimits_m C_{rm}  \bar{C}_{sm}(T_m)_{j,k} = 0 ~ {\rm if} ~  r \neq s .$$
Now we compute
\begin{align*}
\mathrm{Ad_U}(\inner{0,0,0,e_{rs}}) &= \sum\nolimits_{j,a,c} e_{11} \otimes e_{j+1,c+1} \otimes ( e_{22} + e_{33} ) \otimes e_{aa} \otimes ( (T_a)_{j,r} )^* (T_a)_{c,s}
 \\
& + \sum_{j,a,c,p,b} e_{22} \otimes e_{j+1,c+1} \otimes ( e_{22} + e_{33} ) \otimes e_{ap} \otimes ( \beta^{2,j,r}_{22} )_{a,b} ( \beta^{2,c,s}_{22} )^*_{p,b} \;.
\end{align*}
Coinvariance of $M_3(\C)$ gives the following relations for all $j,r,c,s$:
\begin{subequations}
\begin{align}
  \sum_b ( \beta^{2,j,r}_{22} )_{a,b} ( \beta^{2,c,s}_{22} )^*_{p,b} ~ {\rm is}~ 0 ~ {\rm unless} ~ a = p, \label{beta2diagonaljyo1} \\
  \sum_b ( {\bar{\beta}}^{2,j,r}_{11} )_{a,b} ( \beta^{2,c,s}_{11} )_{p,b} ~ {\rm is}~ 0 ~ {\rm unless} ~ a = p, \label{beta2diagonaljyo2} \\
  ( T_a)_{j,r}^* ( T_a)_{c,s} = ( T_b)_{j,r}^* ( T_b)_{c,s}  ~ \forall a,b\,, \label{beta2diagonaljyo3} \\[4pt]
  ( T_a)_{j,r}^* ( T_a)_{c,s} = \sum_b ( \beta^{2,j,r}_{22} )_{a,b} ( \beta^{2,c,s}_{22} )^*_{a,b} ~ \forall a\,, \label{beta2diagonaljyo4} \\
  \sum_b ( \beta^{2,j,r}_{22} )_{a,b} ( \beta^{2,c,s}_{22} )^*_{a,b} = \sum_b ( {\bar{\beta}}^{2,j,r}_{11} )_{a^{\prime},b} ( \beta^{2,c,s}_{11} )_{a^{\prime},b} ~ \forall a,a^{\prime}\,. \label{beta2diagonaljyo5}
 \end{align}
\end{subequations}
However, it turns out that \eqref{beta2diagonaljyo3} is the only new information.
Indeed, \eqref{beta2diagonaljyo1} and \eqref{beta2diagonaljyo2} are consequences of the facts that $ \beta^{2,i,j}_{11} $ is diagonal ( \eqref{beta2diagonaljyo} )  and $ \beta^{2,j,r}_{22} = \bar{\beta^{2,j,r}_{11}}$ ( part 2. of Lemma \ref{lemma:5} ). The equation \eqref{beta2diagonaljyo4}  follows again from \eqref{beta2diagonaljyo}. Finally \eqref{beta2diagonaljyo5} follows from \eqref{beta2diagonaljyo3} and \eqref{beta2diagonaljyo4} taken together.
The equations \eqref{beta2diagonaljyo1} - \eqref{beta2diagonaljyo5} show that $ \mathrm{Ad_U}(\inner{0,0,0,e_{rs}}) $ is given by \eqref{eq:56c}. This completes the proof.
\end{proof}

\smallskip

\noindent
We are now in the position to prove Proposition \ref{prop:main}, i.e.~that the universal object in
category $\catA$ is the CQG given in Lemma~\ref{prop:mainA}. with corepresentation $U$ as in \eqref{eq:Ufourth}.

\smallskip

\begin{proof}[\bf Proof of Proposition \ref{prop:main}]~\\[2pt]
The proof is in two steps: 1. we need to prove that the CQG in Lemma~\ref{prop:mainA}
with corepresentation \eqref{eq:Ufourth} is an object of the category $\catA$ and 2. we need to prove that this object is universal.

\medskip

\noindent{\bf 1.}
First we notice that the operator $U$ in \eqref{eq:Ufourth} is indeed a unitary corepresentation: the unitaries/biunitaries $x_k$, $x_0x_k$, $T_m$, $x_0^*T_m$, $V$ and their ``bar'' define unitary corepresentations due to \eqref{eq:zzz}, and they coact on orthogonal subspaces of $H_F$ in \eqref{eq:Ufourth} so that $U$ is an orthogonal direct sum of unitary corepresentations.

Since our $U$ is of the form \eqref{eq:Uprime}, by Lemma \ref{lemma:appendix1} it
satisfies the compatibility conditions with $\gamma_F$ and $J_F$.

Since $U$ is of the form \eqref{eq:Usecond}, with parameters
\begin{align*}
\alpha^1_{11} &:= \sum^n_{k = 1} e_{kk} \otimes x_0 x_k, &
\alpha^1_{22} &:=  \sum^n_{i,j = 1} e_{ij} \otimes  V_{ij}, &
\alpha^2_{11} &:=(\alpha^2_{22})^*:= \sum^n_{k = 1} e_{kk} \otimes x_k, \\
\beta^{1,j,k}_{22} &:= \sum^n_{m= 1} e_{mm} \otimes (T_m)_{j,k}, &
\beta^{1,j,k}_{11} &:= ( \beta^{1,j,k}_{22} )^*, &
\beta^{2,j,k}_{11} &:= \sum^n_{m = 1} e_{mm} \otimes  x_0^* (T_m)_{j,k} , \\
&& \hspace{-2cm}\alpha^i_{j_1,k_1}:=\; & \beta^{i,j_0,k_0}_{j_1,k_1} :=0 ,  &&\hspace{-1.5cm}\mathrm{if}\;j_1\neq k_1 ,
\end{align*}
satisfying 1. -- 4. of Lemma \ref{lemma:5}, by Lemmas \ref{lemma:appendix2} and \ref{lemma:5} it follows that $U$ commutes with $D_F$.

Since the parameters defined above satisfy the conditions in Lemma \ref{lemma:Utenth} too, we
have that the adjoint coaction preserves $B_F$ and then our CQG with corepresentation \eqref{eq:Ufourth} is an object of the category $\catA$.

\medskip

\noindent{\bf 2.}
Now we pass to universality.
From Lemmas \ref{lemma:appendix1}-\ref{lemma:Utenth} it follows that any object $(Q,U)$
in the category $\catA$ must be generated by the matrix entries of a corepresentation $U$ of the form
\eqref{eq:Ufourth}, with matrix entries satisfying \eqref{eq:qisot}. In particular \eqref{eq:qisotC} coincides
with \eqref{eq:UtenthC}, \eqref{eq:qisotB} coincides with \eqref{eq:UtenthB}, and \eqref{eq:qisotA} coincides with
point 3. of Lemma \ref{lemma:5} after the parameter substitution in (\ref{eq:UtenthD}--\ref{eq:UtenthA})
and after renaming $V$ the matrix $\alpha^1_{22}$. 
Different summands in \eqref{eq:Ufourth} coact on orthogonal subspaces of $H_F$: hence from unitarity
of $U$ we deduce the unitarity of the $x_k$'s and of the matrices $T_m,V$ and their ``bar'', i.e.~they must
be biunitary.
This proves that any object in the
category is a quotient of the CQG in Prop.~\ref{prop:main}.
\end{proof}

\begin{center}
\textsc{Acknowledgments}
\end{center}
\vspace{-4pt}
We would like to thank Prof.~J.W.~Barrett and Prof.~A.H.~Chamseddine for useful conversations and comments.


\noindent{\small
\textit{Email addresses:} \texttt{jyotishmanb@gmail.com}, \texttt{dandrea@sissa.it}, \texttt{dabrow@sissa.it}.}


\begin{thebibliography}{99}

\footnotesize\itemsep=0pt

\bibitem{Ahm02}
Q.R. Ahmad et al. (SNO Collaboration),
\textit{Direct Evidence for Neutrino Flavor Transformation from Neutral-Current Interactions in the Sudbury Neutrino Observatory}, Phys. Rev. Lett. \textbf{89} (2002), 011301.

\bibitem{Ban97}
T. Banica, \textit{Le groupe quantique compact libre $U(n)$}, Commun. Math. Phys. \textbf{190} (1997), 143--172.

\bibitem{Ban05a}  T. Banica, \textit{Quantum automorphism groups of small metric spaces}, Pacific J. Math. \textbf{219} (2005), 27--51.

\bibitem{Ban05b} T. Banica,  \textit{Quantum automorphism groups of homogeneous graphs},  J. Funct. Anal. \textbf{224} (2005), 243--280.

\bibitem{BV09}
T. Banica and R. Vergnioux,
\textit{Invariants of the half-liberated orthogonal group},
to appear in Ann. Inst. Fourier,
arxiv:0902.2719v2 [math.QA].

\bibitem{BGS08}
J.~Bhowmick, D.~Goswami and A.~Skalski, \textit{Quantum Isometry Groups of $0$-Dimensional Manifolds},
to appear in Trans. AMS, arXiv:0807.4288v1 [math.OA].

\bibitem{BG09}
J. Bhowmick and D. Goswami, \textit{Quantum Group of Orientation preserving
  Riemannian Isometries}, J. Funct. Anal. \textbf{257} (2009), 2530--2572.

\bibitem{BG10} J. Bhowmick and D. Goswami, \textit{Quantum isometry groups of the Podles spheres}, J. Funct. Anal. \textbf{258} (2010), 2937--2960.

\bibitem{BG09b}
J.~Bhowmick and D.~Goswami, \textit{Some counterexamples in the theory of quantum isometry groups},
to appear in Lett. Math. Phys., arXiv:0910.4713v1 [math.OA].

\bibitem{BS10}
J.~Bhowmick and A.~Skalski, \textit{Quantum isometry groups of noncommutative manifolds associated to group $C^*$-algebras},
to appear in J. Geom. Phys., arXiv:1002.2551v2 [math.OA].

\bibitem{Bic03}
J. Bichon, \textit{Quantum automorphism groups of finite graphs},  Proc. Amer. Math. Soc. \textbf{131} (2003), no. 3, 665--673.

\bibitem{CC97}
A.H. Chamseddine and A.~Connes, \textit{The Spectral Action Principle}, Commun. Math. Phys. \textbf{186} (1997), 731--750.

\bibitem{CC08}
A.H. Chamseddine and A.~Connes, \textit{Why the Standard Model}, J. Geom. Phys. \textbf{58} (2008), 38--47.

\bibitem{CCM06}
A.H. Chamseddine, A.~Connes and M.~Marcolli, \emph{{Gravity and the standard
  model with neutrino mixing}}, Adv. Theor. Math. Phys. \textbf{11} (2007),
  991--1090.

\bibitem{Con94}
A.~Connes, \textit{Noncommutative Geometry}, Academic Press, 1994.

\bibitem{Con95}
A.~Connes, \emph{Noncommutative geometry and reality}, J.~Math. Phys. \textbf{36} (1995), 6194--6231.

\bibitem{Con98}
A.~Connes, \textit{Noncommutative differential geometry and the structure
of space-time}, Proceedings of the Symposium on Geometry, Huggett,
S.A. (ed.) et al., pp.~49-80, Oxford Univ. Press, Oxford UK 1998.

\bibitem{Con06}
A.~Connes, \textit{Noncommutative geometry and the Standard Model with neutrino
  mixing}, JHEP \textbf{11} (2006), 081.

\bibitem{Con08}
A.~Connes, \textit{On the spectral characterization of manifolds}, 2008,
  arXiv:0810.2088v1 [math.OA].

\bibitem{CM08}
A.~Connes and M.~Marcolli, \textit{Noncommutative geometry, quantum fields and
  motives}, Colloquium Publications, vol.~55, AMS, 2008.

\bibitem{Coq97}
R. Coquereaux, \textit{On the finite dimensional quantum group $M_3\oplus
(M_{2|1}(\Lambda^2))_0$}, Lett. Math. Phys. \textbf{42} (1997), 309--328.

\bibitem{DDLW07}
F.~D'Andrea, L.~D{\k a}browski, G.~Landi and E.~Wagner, \textit{Dirac operators
  on all Podle\'s spheres}, J. Noncomm. Geom. \textbf{1} (2007), 213--239.

\bibitem{DDL06}
F.~D'Andrea, L.~D{\k a}browski and G.~Landi, \textit{The Isospectral Dirac
  Operator on the $4$-dimensional Orthogonal Quantum Sphere}, Commun. Math.
  Phys. \textbf{279} (2008), 77--116.

\bibitem{DLPS05}
L.~D{\k a}browski, G.~Landi, M.~Paschke and A.~Sitarz, \textit{The spectral
  geometry of the equatorial Podle\'s sphere}, Comptes Rendus Acad. Sci. Paris
  \textbf{340} (2005), 819--822.

\bibitem{DLS05}
L.~D{\k a}browski, G.~Landi, A.~Sitarz, W.~van Suijlekom and J.C. V\'arilly,
  \textit{The Dirac operator on $SU_q(2)$}, Comm. Math. Phys. \textbf{259}
  (2005), 729--759.

\bibitem{DNS98}
L. Dabrowski, F. Nesti and P. Siniscalco, \textit{A Finite Quantum Symmetry of $M(3,\C)$},
Int. J. Mod. Phys. \textbf{A13} (1998), 4147--4162.

\bibitem{Fuk98}
Y. Fukuda et al. (Super-Kamiokande Collaboration),
\textit{Evidence for Oscillation of Atmospheric Neutrinos},
Phys. Rev. Lett. \textbf{81} (1998), 1562--1567.

\bibitem{Gos07}
D. Goswami, \textit{Quantum Group of Isometries in Classical and Noncommutative Geometry},
  Commun. Math. Phys. \textbf{285} (2009), 141--160.

\bibitem{Gos10}
D. Goswami, \textit{Quantum Isometry Group for Spectral Triples with Real Structure},
  SIGMA \textbf{6} (2010), 007.

\bibitem{Kas95}
D. Kastler, \textit{Regular and adjoint representation of $SL_q(2)$ at third root of
unit}, CPT internal report (1995).

\bibitem{LMMS96}
F. Lizzi, G. Mangano, G. Miele and G. Sparano,
\textit{Fermion Hilbert space and fermion doubling in the noncommutative geometry approach to gauge theories},
Phys. Rev. \textbf{D55} (1997), 6357--6366.

\bibitem{MD98}
A.~Maes and A.~Van Daele, \textit{Notes on compact quantum groups}, Nieuw Arch. Wisk. \textbf{16}  (1998),  73--112.

\bibitem{Pod95} P. Podles, \textit{Symmetries of quantum spaces. Subgroups and quotient spaces of quantum $SU(2)$ and $SO(3)$ groups}, Comm. Math. Phys. \textbf{170} (1995), 1-20.

\bibitem{Sol10}
P.M.~So{\l}tan, \textit{Quantum $SO(3)$ groups and quantum group actions on $M_2$},
J. Noncommut. Geom. \textbf{4} (2010), 1--28.

\bibitem{DW96}
A.~Van Daele and S.~Wang, \textit{Universal quantum groups}, International J. Math. \textbf{7} (1996), 255--264.

\bibitem{Wan95}
S.~Wang, \textit{Free products of compact quantum groups}, Comm. Math. Phys. \textbf{167} (1995), no. 3, 671--692.

\bibitem{Wan98a}
S. Wang, \textit{Quantum Symmetry Groups of Finite Spaces}, Commun. Math. Phys.
  \textbf{195} (1998), 195--211.

\bibitem{Wan98b}
S. Wang, \textit{Structure and Isomorphism Classification of Compact Quantum Groups $A_u(Q)$ and $B_u(Q)$}, J. Operator Theory \textbf{48} (2002), 573--583.

\bibitem{Wan99}
S.~Wang, \textit{Ergodic actions of universal quantum groups on operator algebra}, Comm. Math. Phys. \textbf{203} (1999), no. 2, 481--498.

\bibitem{Wor87}
S.L.~Woronowicz, \textit{Compact matrix pseudogroups},
   Commun. Math. Phys. \textbf{111} (1987), 613--665.

\bibitem{Wor95}
S.L.~Woronowicz, \textit{Compact quantum groups}, in Sym\'etries quantiques
(Les Houches, 1995), pp.~845--884, edited by A. Connes et al., Elsevier, Amsterdam, 1998.

\end{thebibliography}
\end{document}